\documentclass[10pt]{amsart}
\usepackage{latexsym}
\usepackage{amscd}
\usepackage{amssymb}
\usepackage{graphicx}
\usepackage{epstopdf}
\usepackage{amsmath}
\usepackage{textcomp}

\usepackage{cancel}
\usepackage{tikz}
\usetikzlibrary{decorations.markings}
\tikzstyle{vertex}=[circle, draw, inner sep=0pt, minimum size=6pt]
\newcommand{\vertex}{\node[vertex]}

\usepackage{caption}
\usepackage{subcaption}
\usepackage{amssymb, amscd, amsmath,amsthm,amsfonts,epsfig, enumerate}
\usepackage{epsfig}
\usepackage{fullpage}
\usepackage{listings}
\usepackage{float}
\usepackage{epstopdf}
\usepackage{color}

\def\NZQ{\Bbb}               

\def\ZZ{{\NZQ Z}}

\def\B'c{{\mathcal{B'}}}
\def\U'c{{\mathcal{U'}}}
\def\opn#1#2{\def#1{\operatorname{#2}}} 
\opn\chara{char}
\opn\length{\ell}
\opn\projdim{proj\,dim}
\opn\injdim{inj\,dim}
\opn\ini{in}
\opn\rank{rank}
\opn\depth{depth}
\opn\sdepth{sdepth}
\opn\indmat{indmat}
\opn\cochord{cochord}
\opn\pdim{pdim}
\opn\height{ht}
\opn\embdim{emb\,dim}
\opn\codim{codim}

\opn\Tr{Tr}
\opn\bigrank{big\,rank}
\opn\superheight{superheight}\opn\lcm{lcm}
\opn\trdeg{tr\,deg}%
\opn\reg{reg}
\opn\lreg{lreg}
\opn\set{set}
\opn\supp{Supp}
\opn\shad{Shad}
\opn\div{div}
\opn\Div{Div}
\opn\cl{cl}
\opn\Cl{Cl}
\opn\Spec{Spec}
\opn\Supp{Supp}
\opn\supp{supp}
\opn\Sing{Sing}
\opn\Ass{Ass}
\opn\Min{Min}
\opn\size{size}
\opn\bigsize{bigsize}
\opn\lex{lex}
\opn\Ann{Ann}
\opn\Rad{Rad}
\opn\Soc{Soc}
\opn\Ker{Ker}
\opn\Coker{Coker}
\opn\Im{Im}
\opn\Hom{Hom}
\opn\Tor{Tor}
\opn\Ext{Ext}
\opn\End{End}
\opn\Aut{Aut}
\opn\id{id}

\opn\nat{nat}
\opn\GL{GL}
\opn\SL{SL}
\opn\mod{mod}
\opn\ord{ord}
\opn\aff{aff}
\opn\con{conv}
\opn\relint{relint}
\opn\st{st}
\opn\lk{lk}
\opn\cn{cn}
\opn\core{core}
\opn\vol{vol}
\opn\gr{gr}

\def\pot#1#2{#1[\kern-0.28ex[#2]\kern-0.28ex]}

\opn\dirlim{\underrightarrow{\lim}}
\opn\invlim{\underleftarrow{\lim}}
\let\tensor=\otimes

\def\pnt{{\raise0.5mm\hbox{\large\bf.}}}

\def\Implies{\ifmmode\Longrightarrow \else
	\unskip${}\Longrightarrow{}$\ignorespaces\fi}
\def\implies{\ifmmode\Rightarrow \else
	\unskip${}\Rightarrow{}$\ignorespaces\fi}
\def\iff{\ifmmode\Longleftrightarrow \else
	\unskip${}\Longleftrightarrow{}$\ignorespaces\fi}

\let\:=\colon
\newtheorem{Theorem}{Theorem}[section]
\newtheorem{Lemma}[Theorem]{Lemma}
\newtheorem{Corollary}[Theorem]{Corollary}
\newtheorem{Proposition}[Theorem]{Proposition}
\newtheorem{Remark}[Theorem]{Remark}

\let\epsilon=\varepsilon
\let\phi=\varphi
\let\kappa=\varkappa
\numberwithin{equation}{section}

\title{Some Algebraic Invariants of the residue class rings of the edge ideals of perfect semiregular trees}
\author[Bakhtawar Shaukat]{Bakhtawar Shaukat}
\address{Bakhtawar Shaukat, School of Natural Sciences, National University of Sciences and Technology Islamabad, Sector H-12, Islamabad Pakistan.}
\email{bakhtawar.shaukat@sns.nust.edu.pk}
\author[Ahtsham ul Haq]{Ahtsham ul Haq}
\address{Ahtsham ul Haq, School of Natural Sciences, National University of Sciences and Technology Islamabad, Sector H-12, Islamabad Pakistan.}
\email{ahtsham2192@gmail.com}

\author[Muhammad Ishaq]{Muhammad Ishaq}
\address{Muhammad Ishaq, School of Natural Sciences, National University of Sciences and Technology Islamabad, Sector H-12, Islamabad Pakistan.}
\email{ishaq\_maths@yahoo.com}
\begin{document}
	\maketitle
	\begin{abstract}
		Let $S$ be a polynomial algebra over a field. If $I$ is the edge ideal of a perfect semiregular tree, then we give precise formulas for values of depth, Stanley depth, projective dimension, regularity and Krull dimension of $S/I$.
	\end{abstract}
	
	\section*{Introduction}
	Let $S=K[x_1, \dots , x_n]$ be a polynomial ring over a field $K$. Let $G=(V(G),E(G))$ be a graph with $V(G)=\{x_1,\dots,x_n\}$ and edge set $E(G)$. Any ideal $I$ of $S$ generated by squarefree quadratic monomials can be viewed as the so-called \textit{edge ideal} $I(G)$ of the graph $G$ whose edges are the sets formed by two variables $x_i$, $x_j$ such that $x_ix_j$ is a generator of $I$. Let $M$ be finitely generated $\mathbb{Z}^n$-graded $S$-module. The $K$-subspace $uK[Z]$ which is generated by all elements of the form $uy$ where $u$ is a homogeneous element in $M,$ $y$ is a monomial in $K[Z]$ and $Z\subseteq\{x_1,\dots,x_n\}$. If $uK[Z]$ is a free $K[Z]$-module then it is called a Stanley space of dimension $|Z|$. A decomposition $\mathcal{D}$ of $K$-vector space $M$ as a finite direct sum of Stanley spaces is called a Stanley decomposition of $M$. Let $$\mathcal{D} \, : \, M = \bigoplus_{j=1}^r u_j K[Z_j],$$
	the Stanley depth of $\mathcal{D}$ is  $\sdepth(\mathcal{D})=\min\{|Z_j|:j=1,2,\dots,r\}$. The number
	$$\sdepth(M):=\max \{\sdepth(\mathcal{D})\,:\mathcal{D}\text{ is a Stanley decomposition of} \, M\},$$ is called the \textit{Stanley depth} of $M$. If $\mathfrak{m}:=(x_1,\dots,x_n)$, then the \textit{depth} of $M$ is defined to be the common length of all maximal $M$-sequences in $\mathfrak{m}$. 
In 1982	Stanley conjectured in  \cite{20} that $\sdepth{M}\geq \depth{M}$.  
	Later on, this conjecture was disproved in $2016$ by Duval et al. \cite{21}.   
	Herzog et al. showed in \cite{herz} that the invariant Stanley depth of $J/I$ is combinatorial in nature, where $I\subset J \subset S$ are monomial ideals. But interestingly it shares some properties and bounds with homological invariant depth; see for instance \cite{MC,fakh,fouliii,herz,susanmoreytree,AR1}. 
	
	For a finitely generated $\ZZ^n$-graded S-module $M$, there exists
an exact sequence of minimal possible length, called a minimal standard graded free resolution of M:
	$$0\longrightarrow\ \bigoplus_{j\in \ZZ} S(-j)^{\beta_{r,j}(M)} \longrightarrow \bigoplus_{j\in \ZZ}  S(-j)^{\beta_{r-1,j}(M)}
	\longrightarrow \dots 	\longrightarrow\bigoplus_{j\in \ZZ}  S(-j)^{\beta_{0,j}(M)}\longrightarrow M \longrightarrow\ 0.$$
	The numbers $\beta _{i,j}(M)$'s are called $(i,j)$-th graded Betti numbers of $M$; this number equals the number of minimal generators of degree $j$ in the $i^{th}$ syzygy module of $M$.
	There are certain invariants associated with a minimal graded free resolution $ M $. The \textit{Castelnuovo-Mumford regularity} (or simply \textit{regularity}) of a module $M$ denoted by $\reg(M)$, is given by
	\begin{equation*}
		\reg(M)=\max\{j-i:\beta_{i,j}(M)\neq 0\}.
	\end{equation*}
	The \textit{projective dimension} of module $M$ denoted by $\pdim(M)$, is given by 	
	\begin{equation*}
		\pdim(M)=\max\{i:\beta_{i,j}(M)\neq 0\}.\end {equation*}
		These two invariants measure the size of a resolution. For some more literature related to resolutions we refer the readers to \cite{regnpro,faridi,hermap,zheng}. The interplay of algebraic properties of $I(G)$ and graph-theoretic properties of $G$ is also of great interest; see for instance \cite{ susanmoreytree,moreyvir,book,woodroof}. For values bounds and some other interesting results related to the said invariants we refer the readers to  \cite{regul,naem,hibii,IQ,susanmoreytree,circulent}.
		
		The \textit{degree} of a vertex of a graph is the number of edges that are incident to that vertex. 
		Let $n\geq 2$. A \textit{path} of length $n-1$ on vertices $\{x_1, \dots,x_n\}$ is a
		graph denoted by $P_n$ such that $E(P_n)= \{\{x_i,x_{i+1}\}: 1\leq i\leq n-1\}$.
		A graph $T$ is said to be a \textit{tree} if there exists a unique path between any two vertices of
		$T$. A vertex of degree $1$ of a graph is called a \textit{pendant vertex} (or \textit{leaf}). A tree is called \textit{semiregular} when all of its non-pendant vertices have the same degree. 
		A \textit{rooted tree} is a tree in which one vertex has been designated the root.  The \textit{distance} between two vertices $x_i$ and $x_j$ in a graph is the shortest path between $x_i$ and $x_j$. An \textit{internal vertex} is a vertex that is not a leaf. If $k\geq2$, then a tree with one internal vertex and $k-1$ leaves incident on it is called a \textit{$k$-star}, we denote a $k$-star by $\mathbb{S}_k$.  
		
		In this paper, we define a new class of semiregular rooted trees. We call a semiregular rooted tree a \textit{perfect semiregular tree} if its all pendent vertices are at the same distance from the root. Let $n\geq2$ and $k\geq 1$. We denote a perfect semiregular tree by $T_{n,k}$, where $k$ and $n$ represent the distance of the pendent vertices from the root and degree of the non-pendent vertices, respectively.  For $n\geq 2$ and $k\geq 1$, a \textit{perfect $n$-ary tree} as defined in \cite{n-ary}, is a rooted tree whose root is of degree $n$, and all other internal vertices (if exist) are of degree $n+1$ and all leaves are at distance $k$ from the root (if $k=1$, then a perfect $n$-ary tree is just $\mathbb{S}_{n+1}$). A perfect $(n-1)$-ary tree is an induced subtree of $T_{n,k}$, we denote a perfect $(n-1)$-ary tree by $T'_{n,k}$.	See Figure 1 for examples of $T'_{n,k}$ and $T_{n,k}.$ 
		\begin{figure}[H]
			\centering
			\begin{subfigure}[b]{0.45\textwidth}
				\centering
				\[\begin{tikzpicture}[x=0.4cm, y=0.4cm]
					\vertex[fill] (1) at (0:0){};
					\vertex[fill] (2) at (240:2){};
					\vertex[fill] (3) at (300:2){};
					
					\vertex[fill] (4) at (225:4) {};
					\vertex[fill] (5) at (255:4) {};
					\vertex[fill] (6) at (285:4) {};
					\vertex[fill] (7) at (315:4) {};
					\vertex[fill] (38) at (217.5:6) {};
					\vertex[fill] (39) at (232.5:6) {};
					\vertex[fill] (40) at (247.5:6) {};
					\vertex[fill] (41) at (262.5:6) {};
					\vertex[fill] (42) at (277.5:6) {};
					\vertex[fill] (43) at (292.5:6) {};
					\vertex[fill] (44) at (307.5:6) {};
					\vertex[fill] (45) at (322.5:6) {};		
					\vertex[fill] (38a) at (213.75:8) {};
					\vertex[fill] (38b) at (221.25:8) {};
					\vertex[fill] (39a) at (228.75:8) {};
					\vertex[fill] (39b) at (236.25:8) {};
					\vertex[fill] (40a) at (243.75:8) {};
					\vertex[fill] (40b) at (251.25:8) {};
					\vertex[fill] (41a) at (258.75:8) {};
					\vertex[fill] (41b) at (266.25:8) {};
					\vertex[fill] (42a) at (273.75:8) {};
					\vertex[fill] (42b) at (281.25:8) {};
					\vertex[fill] (43a) at (288.75:8) {};
					\vertex[fill] (43b) at (296.25:8) {};
					\vertex[fill] (44a) at (303.75:8) {};
					\vertex[fill] (44b) at (311.25:8) {};
					\vertex[fill] (45a) at (318.75:8) {};
					\vertex[fill] (45b) at (326.25:8) {};
					
					\path 
					(1) edge (2)
					(1) edge (3)
					(2) edge (4)
					(2) edge (5)
					(3) edge (6)
					(3) edge (7)
					(38) edge (4)
					(40) edge (5)
					(42) edge (6)
					(44) edge (7)
					(39) edge (4)
					(41) edge (5)
					(43) edge (6)
					(45) edge (7)
					(38) edge (38a)
					(40) edge (40a)
					(42) edge (42a)
					(44) edge (44a)
					(39) edge (39a)
					(41) edge (41a)
					(43) edge (43a)
					(45) edge (45a)
					(38) edge (38b)
					(40) edge (40b)
					(42) edge (42b)
					(44) edge (44b)
					(39) edge (39b)
					(41) edge (41b)
					(43) edge (43b)
					(45) edge (45b);
				\end{tikzpicture}\]
			\end{subfigure}
			\hfill
			\begin{subfigure}[b]{0.5\textwidth}
				\centering
				\[\begin{tikzpicture}[x=0.3cm, y=0.3cm]
					\vertex[fill] (0) at (0:0){};
					\vertex[fill] (1) at (15:3) {};
					\vertex[fill] (2) at (105:3) {};
					\vertex[fill] (3) at (195:3) {};	
					\vertex[fill] (3a) at (285:3) {};	
					\vertex[fill] (10) at (315:6)  {};
					\vertex[fill] (11) at (345:6) {};	
					\vertex[fill] (12) at (15:6) {};
					\vertex[fill] (13) at (45:6) {};
					\vertex[fill] (14) at (75:6) {};
					\vertex[fill] (15) at (105:6){};
					\vertex[fill] (16) at (135:6) {};
					\vertex[fill] (17) at (165:6) {};
					\vertex[fill] (18) at (195:6) {};
					\vertex[fill] (19) at (225:6) {};
					\vertex[fill] (20) at (255:6) {};
					\vertex[fill] (21) at (285:6) {};
					\vertex[fill] (22) at (307.5:9) {};
					\vertex[fill] (a) at (315:9) {};
					\vertex[fill] (23) at (322.5:9) {};
					\vertex[fill] (24) at (337.5:9) {};
					\vertex[fill] (b) at (345:9) {};
					\vertex[fill] (25) at (352.5:9) {};
					\vertex[fill] (26) at (7.5:9) {};
					\vertex[fill] (c) at (15:9) {};
					\vertex[fill] (27) at (22.5:9) {};
					\vertex[fill] (28) at (37.5:9) {};
					\vertex[fill] (d) at (45:9) {};
					\vertex[fill] (29) at (52.5:9) {};
					\vertex[fill] (30) at (67.5:9) {};
					\vertex[fill] (e) at (75:9) {};
					\vertex[fill] (31) at (82.5:9) {};
					\vertex[fill] (32) at (97.5:9){};
					\vertex[fill] (f) at (105:9) {};
					\vertex[fill] (33) at (112.5:9) {};
					\vertex[fill] (34) at (127.5:9) {};
					\vertex[fill] (g) at (135:9) {};
					\vertex[fill] (35) at (142.5:9) {};
					\vertex[fill] (36) at (157.5:9) {};
					\vertex[fill] (h) at (165:9) {};
					\vertex[fill] (37) at (172.5:9) {};
					\vertex[fill] (38) at (187.5:9) {};
					\vertex[fill] (i) at (195:9) {};
					\vertex[fill] (39) at (202.5:9) {};
					\vertex[fill] (40) at (217.5:9) {};
					\vertex[fill] (j) at (225:9) {};
					\vertex[fill] (41) at (232.5:9) {};
					\vertex[fill] (42) at (247.5:9) {};
					\vertex[fill] (k) at (255:9) {};
					\vertex[fill] (43) at (262.5:9) {};
					\vertex[fill] (44) at (277.5:9) {};
					\vertex[fill] (l) at (285:9) {};
					\vertex[fill] (45) at (292.5:9) {};	
					
					\path 
					(0) edge (1)
					(0) edge (2)
					(0) edge (3)
					(0) edge (3a)
					(11) edge (1)
					(14) edge (2)
					(17) edge (3)
					(20) edge (3a)
					(12) edge (1)
					(15) edge (2)
					(18) edge (3)
					(21) edge (3a)
					(13) edge (1)
					(16) edge (2)
					(19) edge (3)
					(10) edge (3a)
					(22) edge (10)
					(24) edge (11)
					(26) edge (12)
					(28) edge (13)
					(30) edge (14)
					(32) edge (15)
					(34) edge (16)
					(36) edge (17)
					(38) edge (18)
					(40) edge (19)
					(42) edge (20)
					(44) edge (21)
					(a) edge (10)
					(b) edge (11)
					(c) edge (12)
					(d) edge (13)
					(e) edge (14)
					(f) edge (15)
					(g) edge (16)
					(h) edge (17)
					(i) edge (18)
					(j) edge (19)
					(k) edge (20)
					(l) edge (21)
					(23) edge (10)
					(25) edge (11)
					(27) edge (12)
					(29) edge (13)
					(31) edge (14)
					(33) edge (15)
					(35) edge (16)
					(37) edge (17)
					(39) edge (18)
					(41) edge (19)
					(43) edge (20)
					(45) edge (21);
				\end{tikzpicture}\]
			\end{subfigure}
			\caption{From left to right perfect $2$-ary tree $T'_{3,4}$ and  perfect semiregular tree $T_{4,3}.$}
		\end{figure}
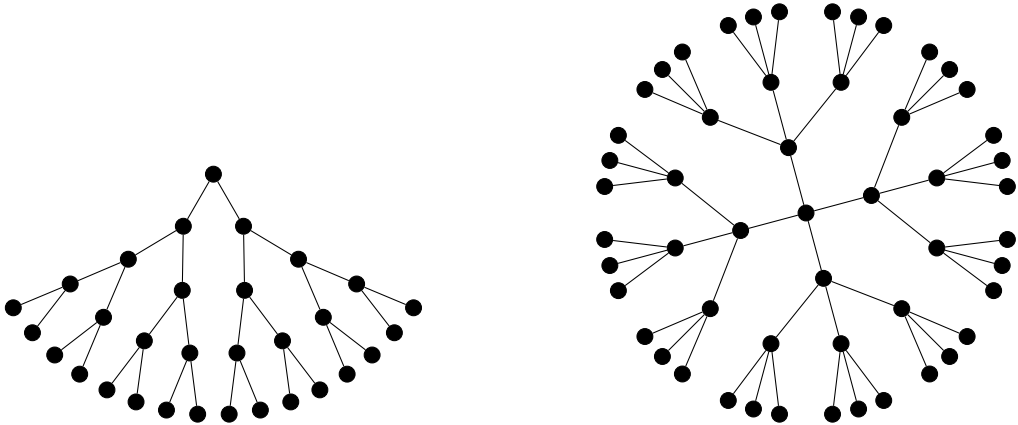
		
	Recently, several types of research have been conducted to determine the precise values and bounds for several algebraic invariants for different classes of monomial ideals and their residue class rings; see for instance \cite{zahidishaq,MI,susanmoreytree,andi,circulent}.
		This paper aims to find the precise formulas for the values of the algebraic invariants depth, Stanley depth, regularity, projective dimension and Krull dimension of $K[V(T_{n,k})]/I({T_{n,k}})$. 
		It is worth mentioning that for computations of the said invariants for $K[V(T_{n,k})]/I(T_{n,k})$ the module $K[V(T'_{n,k})]/I(T'_{n,k})$ plays a vital role.
		It is easy to see that $T_{2,k}=P_{2k+1}$ and $T_{n,1}=\mathbb{S}_{n+1}$, thus our results also complement the previous work on depth, Stanley depth, regularity and projective dimension  of $K[V(P_n)]/I(P_n)$ and $K[V(\mathbb{S}_n)]/I(\mathbb{S}_n)$; see \cite{AA,regnpro,susanmoreytree, alinsdepth}.  We gratefully acknowledge the use of CoCoA \cite{cocoa}, Macaulay 2 \cite{macaulay} and MATLAB\,\textsuperscript{\tiny\textregistered}.
		\section{Preliminaries}
		In this section,  a few basic notions from  Graph Theory and auxiliary results from  Commutative Algebra are reviewed. For a graph $G$, a subset $W$ of $V(G)$ is called an \textit{independent set} if no two vertices in $W$ are adjacent. The carnality of a largest independent
		set of $G$ is called the $independence \ number$ of $G$. A \textit{subgraph} $H$ of a graph $G$, written as $H \subseteq G$, is a graph such that $V(H) \subseteq V(G)$ and  $E(H) \subseteq E(G)$.
		For a subset $ U\subseteq V(G)$, an \textit{induced subgraph} of $G$ is a graph $G':=(U, E( G'))$, such that $E( G')= \{\{x_i,x_j\}\in E(G): \{x_i,x_j\}\subseteq U\} $.
		A \textit{matching} in a graph $G$ is a subset $M$ of $E(G)$ in which no two edges are adjacent. An \textit{induced matching} in $G$ is a matching that forms an induced subgraph of $G$. An $induced \ matching\ number$ of $G$ is denoted as $\indmat(G)$ and  defined as
		$$\indmat(G)=\max \{|M|: \, \, \text{$M$ is an induced matching in $G$}\}.$$ 	
		
	If $G$ is a finite simple graph, Katzman proved in \cite[Lemma 2.2]{kat} that $\indmat(G)$ is a lower bound for the regularity of $S/I(G)$ and then Hà et al. proved in \cite[Corollary  6.9]{hanvan} that regularity of $S/I(G)$ is equal to the $\indmat(G)$ if $G$ is a chordal graph. We combine these results in the following lemma.	\begin{Lemma}\label{reg1}
			Let $G$ be a finite simple graph. Then 
				 $\reg(S/I(G))\geq \indmat(G).$ 
				Moreover, If $G$ is a \text{chordal graph, then} $\reg(S/I(G))=\indmat(G) .$ 
				
		\end{Lemma}
	\begin{Lemma}[{\cite[Lemma 1]{dim}}]\label{prok}
			$\dim (S/I(G))=\max\{|W|:\text{ $W$ is an independent set of $G$}\}.$
		\end{Lemma}
\noindent
An interesting property of depth, Stanley depth and regularity is that when we add new variables to the ring then depth and Stanley depth will also increase {\cite[Lemma 3.6]{herz}}, while regularity will remain the same {\cite[Lemma 3.5]{moreyvir}}. These results are summarized in the following lemma. 
		\begin{Lemma}\label{le3}
			Let $J\subset S$ be a monomial ideal, and $\bar{S}=S \tensor_{K} K[x_{n+1}]$ a polynomial ring of $n+1$ variables. Then  $\depth(\bar{S}/J)=\depth(S/J)+1,   \sdepth(\bar{S}/J)=\sdepth(S/J)+1$ and   $\reg({\bar{S}/J})=\reg({S/J}).$  
		\end{Lemma}
\noindent 
It is obvious and well known that $\depth(S)=\sdepth(S)=n$ and $\reg(S)=0$. 
	\begin{Lemma}[{\cite[Theorems 2.6 and 2.7]{AA}}]\label{leAli}
			If $n\geq 2$ and $I(\mathbb{S}_{n})\subset S:=K[V(\mathbb{S}_{n})]$, then $$\depth(S/I(\mathbb{S}_{n}))=\sdepth(S/I(\mathbb{S}_{n}))=1.$$
		\end{Lemma}
\noindent In the next lemma, we recall an inequality from Depth Lemma {\cite[Proposition 1.2.9]{depth}} and a similar result for Stanley depth proved by Rauf {\cite[Lemma 2.2]{AR1}}. 
		\begin{Lemma}\label{rle2}	
			If $0\rightarrow U \rightarrow V \rightarrow W \rightarrow 0$ is a short exact sequence of $\mathbb{Z}^n$-graded $S$-module, then
		$$\depth (V) \geq \min\{\depth(W), \depth(U)\},$$
				 	$$	\sdepth(V)\geq\min\{\sdepth(U), \sdepth(W)\}.$$
		\end{Lemma}

	 Let $1\leq r < n.$ If  $I\subset S_{1}=K[x_{1},\dots, x_{r}]$ and $J\subset S_{2}=K[x_{r+1},\dots, x_{n}]$ are monomial ideals, then by {\cite[Proposition 2.2.20]{book}} we have
		$S/(I+J)\cong S_{1}/I \tensor_{K}S_{2}/J.$
		Thus we get $\depth(S/(I+J))=\depth(S_{1}/I \tensor_{K}S_{2}/J)$  and $\sdepth(S/(I+J))=\sdepth(S_{1}/I \tensor_{K}S_{2}/J).$ Now applying {\cite[Proposition 2.2.21]{book}} for depth, and Lemma {\cite[Theorem 3.1]{AR1}} for Stanley depth, we have the following useful lemma.
		\begin{Lemma}\label{LEMMA1.5}
			$ \depth_{S}(S_{1}/I\tensor_{K}S_{2}/J)= \depth_{S}(S/(I+J))=\depth_{S_{1}}(S_{1}/I)+\depth_{S_{2}}(S_{2}/J)$ and 	$ \sdepth_{S}(S_{1}/I \tensor_{K} S_{2}/J) \geq \sdepth_{S_{1}}(S_{1}/I)+\sdepth_{S_{2}}(S_{2}/J).$
		\end{Lemma}
		
\noindent In the following lemma, we combine two similar results for depth and Stanley depth. For depth the result is proved by Rauf {\cite[Corollary 1.3]{AR1}}, whereas for Stanley depth it is proved by Cimpoeas {\cite[Proposition 2.7]{MC}}.
		\begin{Lemma} \label{Cor7}
			Let $I\subset S$ be a monomial ideal and $u$ be a monomial in $S$ such that $u\notin I.$ Then
			$\depth (S/(I : u))\geq \depth(S/I)$ and $\sdepth (S/(I : u))\geq \sdepth(S/I).$
		\end{Lemma}

		Now we recall some well know lemmas related to regularity. The form in which we state the first lemma is as stated in {\cite[Theorem 4.7]{regul}}. Part (a) and (c) of this lemma are  immediate consequences of {\cite[Corollary 20.19 and Proposition 20.20]{eisenbud}}, whereas part (b) is an immediate consequence of {\cite[	Lemma 2.10]{dao}}.

		\begin{Lemma}\label{regul}
			Let $ I $ be a monomial ideal and $ x $ be a variable of $S$. Then
			
			\begin{itemize}
				\item [(a)]  $ \reg (S/I) = \reg (S/(I:x))+1$, if $ \reg (S/(I:x)) > \reg (S/(I,x)),$
				\item[(b)] 	  $\reg (S/I)\in \{\reg (S/(I,x))+1,\reg (S/(I,x))\},$ if $ \reg (S/(I:x)) = \reg (S/(I,x)),$
				\item[(c)] 	  $\reg (S/I)= \reg (S/(I,x))$ if $ \reg (S/(I:x)) < \reg (S/(I,x)).$
			\end{itemize}
		\end{Lemma}
The following result is proved by Kalai et al. \cite[Theorem 1.4]{kalai} for squarefree monomial ideals which was then generalized by Herzog in \cite[Corollary 3.2]{hergen} for arbitrary monomial ideals. 	

		\begin{Lemma}\label{circulenttineq}
If $I$ and $J$ are monomial ideals of $S$, then $ \reg({S/(I+J)}) \leq \reg(S/I)+\reg(S/J).$
	\end{Lemma}
\noindent Moreover, if $I_1$ and $I_2$ are two edge ideals minimally generated by disjoint sets of variables then Woodroofe proved the following lemma.
\begin{Lemma}[{\cite[Lemma 8]{woodroof}}]\label{circulentteq}
Let $S_1=K[x_1,\dots,x_r]$ and $S_2=K[x_{r+1},\dots,x_n]$ be rings of polynomials and $I_1$ and $I_2$ be edge ideals of $S_1$ and $S_2$, respectively. Then $$\reg({S/(I_{1}S+I_{2}S)})=\reg(S_{1}/I_{1})+\reg(S_{2}/I_{2}).$$
\end{Lemma}	
		
 We discuss here a few terms that will be helpful throughout this paper. Let $k\geq 0$ and $L_{0},L_{1},\dots,L_{k}$, be subsets of $V(T_{n,k})$, and for $a\in \{0,1,\dots,k\}$, $L_a$ consists of all those vertices whose distance from the root vertex of $T_{n,k}$ is $a$. Thus $|L_0|=1$ and for $a\in \{1,2,\dots,k\}$, we have  $|L_a|=n(n-1)^{a-1}$. It is easy to see that $L_{0},L_{1},L_{2},\dots,L_{k}$ partition $V(T_{n,k})$, therefore, $|V(T_{n,k})|=\sum_{a=0}^{k}|L_a|= \frac{n(n-1)^{k}-2}{(n-2)}$. We label the root vertex of $T_{n,k}$ by $x_1^{(0)}$, that is, $L_0=\{x_1^{(0)}\}$ and for $a\geq 1$ the vertices of $L_a$ are labeled as $L_a=\{x_{i}^{(a)}:1\leq i \leq n(n-1)^{a-1} \}$, see Figure 2. Using this labeling $T_{n,0}=(\{x_1^{(0)}\},\emptyset)$ and for $k\geq 1$, 
		$V(T_{n,k})=\{x_1^{(0)}\}\underset{a=1}{\overset{k}{\cup}}\{x_{i}^{(a)}:1\leq i \leq n(n-1)^{a-1} \}$, 
		$E(T_{n,1})= \underset{l=1}{\overset{n}{\bigcup}}\{\{x_{1}^{(0)}, x_{l}^{(1)}\}\}$,
	and for $k\geq 2,$ we have
		\begin{equation*}
			E(T_{n,k})=\underset{l=1}{\overset{n}{\bigcup}}\{\{x_{1}^{(0)}, x_{l}^{(1)}\}\}\underset{a=1}{\overset{k-1}{\bigcup}}\,\,\,\, \underset{i=1}{\overset{n(n-1)^{a-1}}{\bigcup}}\,\,\underset{j=(n-1)i-(n-2)}{\overset{(n-1)i}{\bigcup}}\{\{x_{i}^{(a)},x_{j}^{(a+1)}\}\}.
		\end{equation*}
		
		Let $A:=\{x_1^{(0)}\}\underset{a=1}{\overset{k}{\cup}}\{x_{i}^{(a)}:1\leq i \leq (n-1)^{a} \}$ be a subset of $V(T_{n,k})$ and $H$ be an induced subgraph of $T_{n,k}$ on $A$. It is easy to see that $H=T'_{n,k}$. Let $L'_{0},L'_{1},L'_{2},\dots,L'_{k}$, be subsets of $V(T'_{n,k})$, such that for $a\in\{0,1,2,\dots,k\}$, $L'_a$ consists of all those vertices whose distance from the root vertex of $T'_{n,k}$ is $a$. Thus $|L'_0|=1$ and for $a\in \{1,2,\dots,k\}$ we have $|L'_a|=(n-1)^{a}$. Since $L'_a\subset L_a$, therefore, $L'_{0},L'_{1},L'_{2},\dots,L'_{k}$ partition $V(T'_{n,k})$ and  $|V(T'_{n,k})|=\sum_{a=0}^{k}|L'_a|= \frac{(n-1)^{k+1}-1}{(n-2)}$. Clearly, $x_1^{(0)}$ is the root vertex of $T'_{n,k}$, that is, $L'_0=\{x_1^{(0)}\}$ and for $a\geq 1$,  $L'_a=\{x_{i}^{(a)}:1\leq i \leq (n-1)^{a} \}$ as illustrated in Figure 2. Thus 	$V(T'_{n,k})=\underset{a=0}{\overset{k}{\cup}}\{x_{i}^{(a)}:1\leq i \leq (n-1)^{a} \}$, $E(T'_{n,0})=\emptyset$ and for $k\geq 1$, 
		
		\begin{equation*}
			E(T'_{n,k})=  \underset{a=0}{\overset{k-1}{\bigcup}}\,\,\,\, \underset{i=1}{\overset{(n-1)^{a}}{\bigcup}}\,\,\underset{j=(n-1)i-(n-2)}{\overset{(n-1)i}{\bigcup}}\{\{x_{i}^{(a)},x_{j}^{(a+1)}\}\}.
		\end{equation*} 
			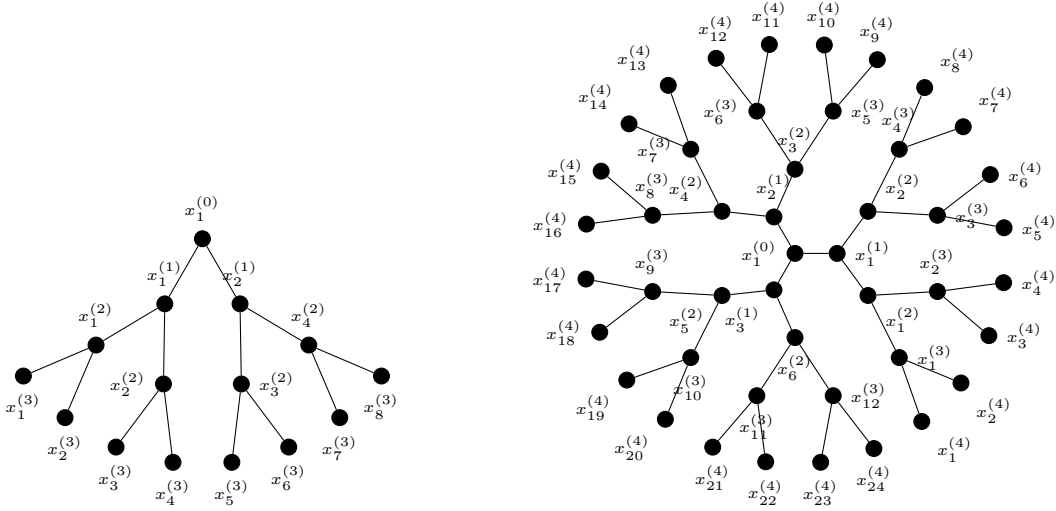
\begin{figure}[H]
			\centering
			\begin{subfigure}[b]{0.45\textwidth}
				\centering
				\[\begin{tikzpicture}[x=0.5cm, y=0.5cm]
					\vertex[fill] (1) at (0:0) [label=above:\tiny${x^{(0)}_{1}}$]{};
					\vertex[fill] (2) at (240:2) [label=above:\tiny${x^{(1)}_{1}}$]{};
					\vertex[fill] (3) at (300:2) [label=above:\tiny${x^{(1)}_{2}}$] {};
					\vertex[fill] (4) at (225:4) [label=above:\tiny${x^{(2)}_{1}}$]{};
					\vertex[fill] (5) at (255:4) [label=left:\tiny${x^{(2)}_{2}}$]{};
					\vertex[fill] (6) at (285:4) [label=right:\tiny${x^{(2)}_{3}}$]{};
					\vertex[fill] (7) at (315:4) [label=above:\tiny${x^{(2)}_{4}}$]{};
					\vertex[fill] (38) at (217.5:6) [label=below:\tiny${x^{(3)}_{1}}$] {};
					\vertex[fill] (39) at (232.5:6) [label=below:\tiny${x^{(3)}_{2}}$]{};
					\vertex[fill] (40) at (247.5:6) [label=below:\tiny${x^{(3)}_{3}}$]{};
					\vertex[fill] (41) at (262.5:6) [label=below:\tiny${x^{(3)}_{4}}$]{};
					\vertex[fill] (42) at (277.5:6) [label=below:\tiny${x^{(3)}_{5}}$]{};
					\vertex[fill] (43) at (292.5:6) [label=below:\tiny${x^{(3)}_{6}}$]{};
					\vertex[fill] (44) at (307.5:6) [label=below:\tiny${x^{(3)}_{7}}$]{};
					\vertex[fill] (45) at (322.5:6) [label=below:\tiny${x^{(3)}_{8}}$]{};
					
					\path 
					(1) edge (2)
					(1) edge (3)
					(2) edge (4)
					(2) edge (5)
					(3) edge (6)
					(3) edge (7)
					(38) edge (4)
					(40) edge (5)
					(42) edge (6)
					(44) edge (7)
					(39) edge (4)
					(41) edge (5)
					(43) edge (6)
					(45) edge (7);
				\end{tikzpicture}\]
				
			\end{subfigure}
			\hfill
			\begin{subfigure}[b]{0.5\textwidth}
				\centering
				\[\begin{tikzpicture}[x=0.28cm, y=0.28cm]
		\vertex[fill] (0) at (0:0) [label=left:\tiny${x^{(0)}_{1}}$]{};
		\vertex[fill] (1) at (0:2) [label=right:\tiny${x^{(1)}_{1}}$]{};
		\vertex[fill] (2) at (120:2) [label=above:\tiny${x^{(1)}_{2}}$]{};
		\vertex[fill] (3) at (240:2) [label=240:\tiny${x^{(1)}_{3}}$]{};
		\vertex[fill] (4) at (330:4) [label=345:\tiny${x^{(2)}_{1}}$]{};
		\vertex[fill] (5) at (30:4) [label=20:\tiny${x^{(2)}_{2}}$]{};
		\vertex[fill] (6) at (90:4) [label=above:\tiny${x^{(2)}_{3}}$]{};
		\vertex[fill] (7) at (150:4) [label=170:\tiny${x^{(2)}_{4}}$]{};
		\vertex[fill] (8) at (210:4) [label=200:\tiny${x^{(2)}_{5}}$]{};
		\vertex[fill] (9) at (270:4) [label=below:\tiny${x^{(2)}_{6}}$]{};
		\vertex[fill] (10) at (315:7) [label=right: \tiny${x^{(3)}_{1}}$]{};
		\vertex[fill] (11) at (345:7) [label=above:\tiny${x^{(3)}_{2}}$]{};	
		\vertex[fill] (12) at (15:7) [label=right :\tiny${x^{(3)}_{3}}$]{};
		\vertex[fill] (13) at (45:7) [label=above:\tiny${x^{(3)}_{4}}$]{};
		\vertex[fill] (14) at (75:7) [label=right:\tiny${x^{(3)}_{5}}$]{};
		\vertex[fill] (15) at (105:7) [label=left:\tiny${x^{(3)}_{6}}$]{};
		\vertex[fill] (16) at (135:7) [label=left:\tiny${x^{(3)}_{7}}$]{};
		\vertex[fill] (17) at (165:7) [label=above:\tiny${x^{(3)}_{8}}$]{};
		\vertex[fill] (18) at (195:7) [label=above:\tiny${x^{(3)}_{9}}$]{};
		\vertex[fill] (19) at (225:7) [label=below:\tiny${x^{(3)}_{10}}$]{};
		\vertex[fill] (20) at (255:7) [label=below:\tiny${x^{(3)}_{11}}$]{};
		\vertex[fill] (21) at (285:7) [label=right:\tiny${x^{(3)}_{12}}$]{};
		\vertex[fill] (10a) at (307:10) [label=315:\tiny${x^{(4)}_{1}}$]{};
		\vertex[fill] (11a) at (337:10) [label=right:\tiny${x^{(4)}_{3}}$]{};	
		\vertex[fill] (12a) at (7:10) [label=right:\tiny${x^{(4)}_{5}}$]{};
		\vertex[fill] (13a) at (37:10) [label=45:\tiny${x^{(4)}_{7}}$]{};
		\vertex[fill] (14a) at (67:10) [label=above:\tiny${x^{(4)}_{9}}$]{};
		\vertex[fill] (15a) at (97:10) [label=above:\tiny${x^{(4)}_{11}}$]{};
		\vertex[fill] (16a) at (127:10) [label=135:\tiny${x^{(4)}_{13}}$]{};
		\vertex[fill] (17a) at (157:10) [label=left:\tiny${x^{(4)}_{15}}$]{};
		\vertex[fill] (18a) at (187:10) [label=left:\tiny${x^{(4)}_{17}}$]{};
		\vertex[fill] (19a) at (217:10) [label=225:\tiny${x^{(4)}_{19}}$]{};
		\vertex[fill] (20a) at (247:10) [label=below:\tiny${x^{(4)}_{21}}$]{};
		\vertex[fill] (21a) at (277:10) [label=below:\tiny${x^{(4)}_{23}}$]{};
		\vertex[fill] (10b) at (322:10) [label=315:\tiny${x^{(4)}_{2}}$]{};
		\vertex[fill] (11b) at (352:10) [label=right:\tiny${x^{(4)}_{4}}$]{};	
		\vertex[fill] (12b) at (22:10) [label=right:\tiny${x^{(4)}_{6}}$]{};
		\vertex[fill] (13b) at (52:10) [label=45:\tiny${x^{(4)}_{8}}$]{};
		\vertex[fill] (14b) at (82:10) [label=above:\tiny${x^{(4)}_{10}}$]{};
		\vertex[fill] (15b) at (112:10) [label=above:\tiny${x^{(4)}_{12}}$]{};
		\vertex[fill] (16b) at (142:10) [label=135:\tiny${x^{(4)}_{14}}$]{};
		\vertex[fill] (17b) at (172:10) [label=left:\tiny${x^{(4)}_{16}}$]{};
		\vertex[fill] (18b) at (202:10) [label=left:\tiny${x^{(4)}_{18}}$]{};
		\vertex[fill] (19b) at (232:10) [label=225:\tiny${x^{(4)}_{20}}$]{};
		\vertex[fill] (20b) at (262:10) [label=below:\tiny${x^{(4)}_{22}}$]{};
		\vertex[fill] (21b) at (292:10) [label=below:\tiny${x^{(4)}_{24}}$]{};
		
		\path 
		(0) edge (1)
		(0) edge (2)
		(0) edge (3)
		(4) edge (1)
		(5) edge (1)
		(6) edge (2)
		(7) edge (2)
		(8) edge (3)
		(9) edge (3)	
		(4) edge (10)
		(4) edge (11)
		(5) edge (12)
		(5) edge (13)
		(6) edge (14)
		(6) edge (15)
		(7) edge (16)
		(7) edge (17)
		(8) edge (18)
		(8) edge (19)
		(9) edge (20)
		(9) edge (21)
		(10) edge (10a)
		(11) edge (11a)
		(12) edge (12a)
		(13) edge (13a)
		(14) edge (14a)
		(15) edge (15a)
		(16) edge (16a)
		(17) edge (17a)
		(18) edge (18a)
		(19) edge (19a)
		(20) edge (20a)
		(21) edge (21a)
		(10) edge (10b)
		(11) edge (11b)
		(12) edge (12b)
		(13) edge (13b)
		(14) edge (14b)
		(15) edge (15b)
		(16) edge (16b)
		(17) edge (17b)
		(18) edge (18b)
		(19) edge (19b)
		(20) edge (20b)
		(21) edge (21b)
		;
	\end{tikzpicture}\]
			\end{subfigure}
			\caption{From left to right perfect $2$-ary tree $T'_{3,3}$ and  perfect semiregular tree $T_{3,4}.$}
		\end{figure} 
		
 Let $\mathbb{M}_{n,k}:=K[V({T'_{n,k}})]/I(T'_{n,k})$, and  $\mathbb{M}^m_{n,k}$ be a $K$-algebra which is the tensor product of $m$ copies of $\mathbb{M}_{n,k}$ over $K$, that is, $\mathbb{M}^m_{n,k}:=\underset{j=1}{\overset{m}{\tensor_{K}}}\mathbb{M}_{n,k}.$ 
In the following remark we address some special cases of $\mathbb{M}_{n,k}^m$ that will be encountered in the proofs of our main theorems.
		\begin{Remark}
			{\em We define  $\mathbb{M}^0_{n,k}:= K$. If we define $I(T'_{n,0})=(0)$, then  $\mathbb{M}_{n,0}\cong K[x_1^{(0)}]$ and   $\mathbb{M}_{n,0}^m\cong \underset{j=1}{\overset{m}{\tensor_{K}}}K[x_1^{(0)}]$. Thus  $\depth(\mathbb{M}_{n,0}^m)=\sdepth(\mathbb{M}_{n,0}^m)=m$. }
		\end{Remark} 	\begin{Lemma}\label{nl}
			Let $n\geq 3$ and $k\geq 1$. Then $\depth(\mathbb{M}^0_{n,k})=\sdepth(\mathbb{M}^0_{n,k})=0$, and for $m\geq1$,   
			$$\depth(\mathbb{M}^m_{n,k})=m\cdot\depth(\mathbb{M}_{n,k}),$$ $$\sdepth(\mathbb{M}^m_{n,k})\geq m\cdot\sdepth(\mathbb{M}_{n,k})$$ and  
			$$\reg(\mathbb{M}^m_{n,k})=m\cdot\reg(\mathbb{M}_{n,k}).$$
		\end{Lemma} 	\begin{proof}
		    The proof follows by using Lemma \ref{LEMMA1.5} and Lemma \ref{circulentteq}.
		\end{proof} 
	 
Let $\mathbb{G}(I)$ denotes the minimal set of monomial generators of monomial ideal $I.$ For a monomial ideal $I$, $\supp{(I)}:=\{x_{i}:x_{i}|u\, \,  \text{for some} \,\,  u \in \mathbb{G}(I)\}.$	
	\begin{Remark}
{\em Let $I$ be a squarefree monomial ideal of $S$ minimally generated by monomials of degree at most $2.$ We associate a graph $G_{I}$ to the ideal $I$ with $V(G_I)=\supp(I)$  and $E(G_I)=\{\{x_{i},x_{j}\}:x_{i}x_{j}\in \mathbb{G}(I)\}.$ Let $x_{t} \in S$ be a variable of the polynomial ring $S$ such that $x_{t} \notin I.$ Then $(I:x_{t})$ and  $(I,x_{t})$ are the monomial ideals of $S$ such that $G_{(I:x_{t})}$ and $G_{(I,x_{t})}$ are subgraphs of $G_{I}.$ See Figure $3$ and $4$ for examples of 	$G_{(I:x_{1}^{(0)})}$,	$G_{(I,x_{1}^{(0)})}$ and $G_{(I: x^{(3)}_{1}x^{(3)}_{2}x^{(3)}_{3}x^{(3)}_{4}x^{(3)}_{5}x^{(3)}_{6}x^{(3)}_{7}x^{(3)}_{8}x^{(3)}_{9}x^{(3)}_{10}x^{(3)}_{11}x^{(3)}_{12})}.$ It is evident form the Figures 3 and 4 that we have the following isomorphisms:
	\begin{equation*}
	   	K[V(T'_{3,3})]/(I(T'_{3,3}):x^{(0)}_{1}) \cong \mathbb{M}^4_{3,1}  \tensor_{K} K[x^{(0)}_{1}],
	\end{equation*}
	$$	K[V(T'_{3,3})]/(I(T'_{3,3}),x^{(0)}_{1})\cong \mathbb{M}^2_{3,2} ,$$ and $$K[V(T_{3,4})]/\big(I(T_{3,4}):x^{(3)}_{1}x^{(3)}_{2}x^{(3)}_{3}x^{(3)}_{4}x^{(3)}_{5}x^{(3)}_{6}x^{(3)}_{7}x^{(3)}_{8}x^{(3)}_{9}x^{(3)}_{10}x^{(3)}_{11}x^{(3)}_{12}\big)\cong K[V(T_{3,1})]/I(T_{3,1}) \tensor_{K} K[L_{3}] .$$}
\end{Remark}

		\begin{figure}[H]
			\centering
			\begin{subfigure}[b]{0.45\textwidth}
				\centering
				\[\begin{tikzpicture}[x=0.5cm, y=0.5cm]
					\vertex[fill] (2) at (240:2) [label=above:\tiny${x^{(1)}_{1}}$]{};
					\vertex[fill] (3) at (300:2) [label=above:\tiny${x^{(1)}_{2}}$] {};
					\vertex[fill] (4) at (225:4) [label=above:\tiny${x^{(2)}_{1}}$]{};
					\vertex[fill] (5) at (255:4) [label=left:\tiny${x^{(2)}_{2}}$]{};
					\vertex[fill] (6) at (285:4) [label=right:\tiny${x^{(2)}_{3}}$]{};
					\vertex[fill] (7) at (315:4) [label=above:\tiny${x^{(2)}_{4}}$]{};
					\vertex[fill] (38) at (217.5:6) [label=below:\tiny${x^{(3)}_{1}}$] {};
					\vertex[fill] (39) at (232.5:6) [label=below:\tiny${x^{(3)}_{2}}$]{};
					\vertex[fill] (40) at (247.5:6) [label=below:\tiny${x^{(3)}_{3}}$]{};
					\vertex[fill] (41) at (262.5:6) [label=below:\tiny${x^{(3)}_{4}}$]{};
					\vertex[fill] (42) at (277.5:6) [label=below:\tiny${x^{(3)}_{5}}$]{};
					\vertex[fill] (43) at (292.5:6) [label=below:\tiny${x^{(3)}_{6}}$]{};
					\vertex[fill] (44) at (307.5:6) [label=below:\tiny${x^{(3)}_{7}}$]{};
					\vertex[fill] (45) at (322.5:6) [label=below:\tiny${x^{(3)}_{8}}$]{};
					
					\path 
					
					(38) edge (4)
					(40) edge (5)
					(42) edge (6)
					(44) edge (7)
					(39) edge (4)
					(41) edge (5)
					(43) edge (6)
					(45) edge (7);
				\end{tikzpicture}\]
				
			\end{subfigure}
			\hfill
			\begin{subfigure}[b]{0.45\textwidth}
				\centering
				\[\begin{tikzpicture}[x=0.5cm, y=0.5cm]
					\vertex[fill] (1) at (0:0) [label=above:\tiny${x^{(0)}_{1}}$]{};
					\vertex[fill] (2) at (240:2) [label=above:\tiny${x^{(1)}_{1}}$]{};
					\vertex[fill] (3) at (300:2) [label=above:\tiny${x^{(1)}_{2}}$] {};
					\vertex[fill] (4) at (225:4) [label=above:\tiny${x^{(2)}_{1}}$]{};
					\vertex[fill] (5) at (255:4) [label=left:\tiny${x^{(2)}_{2}}$]{};
					\vertex[fill] (6) at (285:4) [label=right:\tiny${x^{(2)}_{3}}$]{};
					\vertex[fill] (7) at (315:4) [label=above:\tiny${x^{(2)}_{4}}$]{};
					\vertex[fill] (38) at (217.5:6) [label=below:\tiny${x^{(3)}_{1}}$] {};
					\vertex[fill] (39) at (232.5:6) [label=below:\tiny${x^{(3)}_{2}}$]{};
					\vertex[fill] (40) at (247.5:6) [label=below:\tiny${x^{(3)}_{3}}$]{};
					\vertex[fill] (41) at (262.5:6) [label=below:\tiny${x^{(3)}_{4}}$]{};
					\vertex[fill] (42) at (277.5:6) [label=below:\tiny${x^{(3)}_{5}}$]{};
					\vertex[fill] (43) at (292.5:6) [label=below:\tiny${x^{(3)}_{6}}$]{};
					\vertex[fill] (44) at (307.5:6) [label=below:\tiny${x^{(3)}_{7}}$]{};
					\vertex[fill] (45) at (322.5:6) [label=below:\tiny${x^{(3)}_{8}}$]{};
					
					\path 
					
					(2) edge (4)
					(2) edge (5)
					(3) edge (6)
					(3) edge (7)
					(38) edge (4)
					(40) edge (5)
					(42) edge (6)
					(44) edge (7)
					(39) edge (4)
					(41) edge (5)
					(43) edge (6)
					(45) edge (7);
				\end{tikzpicture}\]
			\end{subfigure}
			\caption{From left to right $G_{(I(T'_{3,3}): x_{1}^{(0)})}$  and $G_{(I(T'_{3,3}),  x_{1}^{(0)})}$.}
		\end{figure}
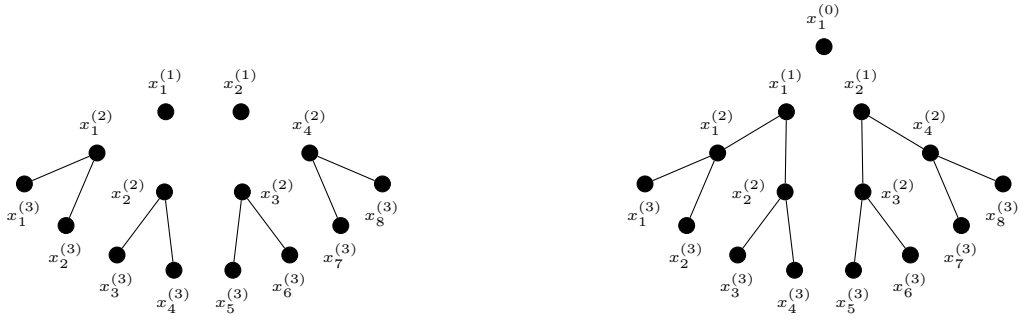 	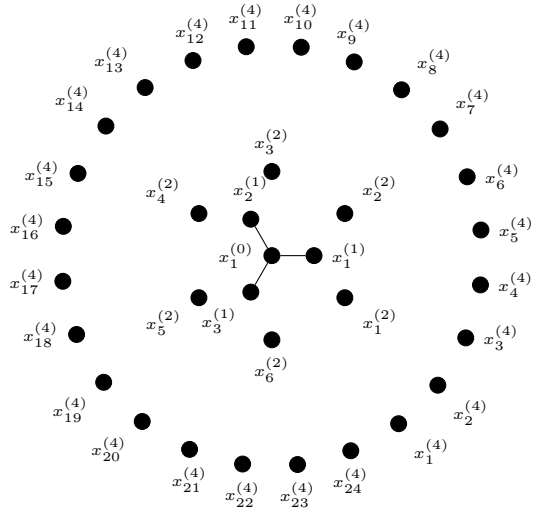
\begin{figure}[H]
	\centering 
	\[\begin{tikzpicture}[x=0.28cm, y=0.28cm]
		\vertex[fill] (0) at (0:0) [label=left:\tiny${x^{(0)}_{1}}$]{};
		\vertex[fill] (1) at (0:2) [label=right:\tiny${x^{(1)}_{1}}$]{};
		\vertex[fill] (2) at (120:2) [label=above:\tiny${x^{(1)}_{2}}$]{};
		\vertex[fill] (3) at (240:2) [label=240:\tiny${x^{(1)}_{3}}$]{};
		\vertex[fill] (4) at (330:4) [label=345:\tiny${x^{(2)}_{1}}$]{};
		\vertex[fill] (5) at (30:4) [label=20:\tiny${x^{(2)}_{2}}$]{};
		\vertex[fill] (6) at (90:4) [label=above:\tiny${x^{(2)}_{3}}$]{};
		\vertex[fill] (7) at (150:4) [label=170:\tiny${x^{(2)}_{4}}$]{};
		\vertex[fill] (8) at (210:4) [label=200:\tiny${x^{(2)}_{5}}$]{};
		\vertex[fill] (9) at (270:4) [label=below:\tiny${x^{(2)}_{6}}$]{};
		\vertex[fill] (10a) at (307:10) [label=315:\tiny${x^{(4)}_{1}}$]{};
		\vertex[fill] (11a) at (337:10) [label=right:\tiny${x^{(4)}_{3}}$]{};	
		\vertex[fill] (12a) at (7:10) [label=right:\tiny${x^{(4)}_{5}}$]{};
		\vertex[fill] (13a) at (37:10) [label=45:\tiny${x^{(4)}_{7}}$]{};
		\vertex[fill] (14a) at (67:10) [label=above:\tiny${x^{(4)}_{9}}$]{};
		\vertex[fill] (15a) at (97:10) [label=above:\tiny${x^{(4)}_{11}}$]{};
		\vertex[fill] (16a) at (127:10) [label=135:\tiny${x^{(4)}_{13}}$]{};
		\vertex[fill] (17a) at (157:10) [label=left:\tiny${x^{(4)}_{15}}$]{};
		\vertex[fill] (18a) at (187:10) [label=left:\tiny${x^{(4)}_{17}}$]{};
		\vertex[fill] (19a) at (217:10) [label=225:\tiny${x^{(4)}_{19}}$]{};
		\vertex[fill] (20a) at (247:10) [label=below:\tiny${x^{(4)}_{21}}$]{};
		\vertex[fill] (21a) at (277:10) [label=below:\tiny${x^{(4)}_{23}}$]{};
		\vertex[fill] (10b) at (322:10) [label=315:\tiny${x^{(4)}_{2}}$]{};
		\vertex[fill] (11b) at (352:10) [label=right:\tiny${x^{(4)}_{4}}$]{};	
		\vertex[fill] (12b) at (22:10) [label=right:\tiny${x^{(4)}_{6}}$]{};
		\vertex[fill] (13b) at (52:10) [label=45:\tiny${x^{(4)}_{8}}$]{};
		\vertex[fill] (14b) at (82:10) [label=above:\tiny${x^{(4)}_{10}}$]{};
		\vertex[fill] (15b) at (112:10) [label=above:\tiny${x^{(4)}_{12}}$]{};
		\vertex[fill] (16b) at (142:10) [label=135:\tiny${x^{(4)}_{14}}$]{};
		\vertex[fill] (17b) at (172:10) [label=left:\tiny${x^{(4)}_{16}}$]{};
		\vertex[fill] (18b) at (202:10) [label=left:\tiny${x^{(4)}_{18}}$]{};
		\vertex[fill] (19b) at (232:10) [label=225:\tiny${x^{(4)}_{20}}$]{};
		\vertex[fill] (20b) at (262:10) [label=below:\tiny${x^{(4)}_{22}}$]{};
		\vertex[fill] (21b) at (292:10) [label=below:\tiny${x^{(4)}_{24}}$]{};
		
		\path 
		(0) edge (1)
		(0) edge (2)
		(0) edge (3)
	
		;
	\end{tikzpicture}\]
			\caption{$G_{(I(T_{3,4})\,:\, x^{(3)}_{1}x^{(3)}_{2}x^{(3)}_{3}x^{(3)}_{4}x^{(3)}_{5}x^{(3)}_{6}x^{(3)}_{7}x^{(3)}_{8}x^{(3)}_{9}x^{(3)}_{10}x^{(3)}_{11}x^{(3)}_{12})}$.}
	\label{fig:4RT}
		\end{figure} 	
		
		\section{Depth, Stanley Depth and Projective dimension}	
		In this section, we compute the depth, projective dimension and Stanley depth of the cyclic module $\mathbb{M}_{n,k}$ and using these results we obtain values for the said invariants of the cyclic module $K[V(T_{n,k})]/I(T_{n,k})$.
		\begin{Lemma}\label{nl1}
			Let $n\geq3$ and  $k\geq 2$. If $k=2$, then  $\depth(\mathbb{M}_{n,2}),\sdepth(\mathbb{M}_{n,2})\leq n-1$, and for $k\geq 3$, $$\depth(\mathbb{M}_{n,k})\leq(n-1)^{k-1}+\depth(\mathbb{M}_{n,k-3}), $$  and $$\sdepth(\mathbb{M}_{n,k})\leq(n-1)^{k-1}+\sdepth(\mathbb{M}_{n,k-3}).$$  
		\end{Lemma}
		\begin{proof}
			We prove the result only for depth as the proof for Stanley depth is similar. Let $S=K[V(T'_{n,k})]$. Since  $u=:x_{1}^{(k-1)}x_{2}^{(k-1)}\cdots x^{(k-1)}_{(n-1)^{k-1}}\notin I(T'_{n,k})$, by Lemma  \ref{Cor7}, $\depth (\mathbb{M}_{n,k})\leq\depth (S/(I(T'_{n,k}):u)).$ We have the following $K$-algebra isomorphism, $S/(I(T'_{n,2}):u)\cong K[L'_1],$ and for $k\geq 3$,  $S/(I(T'_{n,k}):u)\cong K[L'_{k-1}]{\tensor_{K}} \mathbb{M}_{n,k-3}.$ Thus $\depth(S/(I(T'_{n,2}):u))=n-1$ and for $k\geq 3$, by Lemma \ref{le3}   $\depth(S/(I(T'_{n,k}):u))=(n-1)^{k-1}+\depth(\mathbb{M}_{n,k-3}).$ This completes the proof. 
		\end{proof}
		
		\begin{Proposition}\label{leminitial}
			Let $n\geq 3$ and $k\geq 1$. Then 
			
			\begin{equation*}
				\depth(\mathbb{M}_{n,k})=\sdepth(\mathbb{M}_{n,k})=\left\{\begin{matrix}
				\frac{(n-1)^{2}((n-1)^{k}-1)}{(n-1)^{3}-1}+1, & if \quad \quad k\equiv 0\, (\mod\, 3);\\ \\
					\frac{(n-1)^{k+2}-1}{(n-1)^{3}-1}, &if \quad \quad k\equiv 1 \, (\mod\, 3);
					\\	\\\frac{(n-1)^{k+2}-n+1}{(n-1)^{3}-1}, &if \quad \quad k\equiv 2\, (\mod\, 3).
				\end{matrix}\right.
			\end{equation*}
		\end{Proposition}
		\begin{proof}
			First, we will prove the result for depth. If $k=1$, then $\mathbb{M}_{n,1}\cong K[V(\mathbb{S}_n)]/I(\mathbb{S}_n$), by Lemma \ref{leAli} we have $\depth(\mathbb{M}_{n,1})=1$, as required. Let $k\geq2$ and $S=K[V(T'_{n,k})].$ We have the following short exact sequence
			\begin{equation*} 
				0	\longrightarrow\ S/(I(T'_{n,k}):x^{(0)}_{1}) \xrightarrow{\,\,.x^{(0)}_{1}\,}  \mathbb{M}_{n,k} 
				\longrightarrow\ S/(I(T'_{n,k}),x^{(0)}_{1})  \longrightarrow\ 0,
			\end{equation*}	
			It is easy to see that				$S/(I(T'_{n,k}):x^{(0)}_{1}) \cong \mathbb{M}^{(n-1)^2}_{n,k-2}\tensor_{K} K[x^{(0)}_{1}]$	and  $S/(I(T'_{n,k}),x^{(0)}_{1})\cong\mathbb{M}^{n-1}_{n,k-1}$. By Lemma \ref{LEMMA1.5} and Lemma \ref{nl}, we have
			\begin{equation*}		
				\depth (S/(I(T'_{n,k}):x^{(0)}_{1})) = (n-1)^2\depth (\mathbb{M}_{n,k-2})+1,	
			\end{equation*}
			\begin{equation*}
				\depth (S/(I(T'_{n,k}),x^{(0)}_{1})) = (n-1)\depth(\mathbb{M}_{n,k-1}).		\end{equation*}	
			Thus by Lemma \ref{rle2}, we have 
			\begin{equation}\label{eq1}
				\depth(\mathbb{M}_{n,k})\geq \min\{(n-1)^2\depth (\mathbb{M}_{n,k-2})+1,\,(n-1)\depth(\mathbb{M}_{n,k-1})\}.  
			\end{equation}
			If $k=2$, then by Eq. \ref{eq1}  $$\depth(\mathbb{M}_{n,2})\geq \min\{(n-1)^2\depth (\mathbb{M}_{n,0})+1,\,(n-1)\depth(\mathbb{M}_{n,1})\}=\min\{(n-1)^2+1,(n-1)\}=n-1,$$ and by Lemma \ref{nl1}, $\depth(\mathbb{M}_{n,2})\leq n-1$. Thus $\depth(\mathbb{M}_{n,2})= n-1$. This prove the result for $k=2$.\\
			Let $k\geq 3$. For $1\leq i\leq n-1$, let $A_i:=(x^{(1)}_{1}	,x^{(1)}_{2}	,\dots,x^{(1)}_{i})$ and $A_0:=(0)$ be prime ideals. Consider the family of short exact sequences:
			$$	0	\longrightarrow\ S/((I(T'_{n,k}),A_{i-1})	:x^{(1)}_{i}	) \xrightarrow{\,\,.x^{(1)}_{i}\,}  S/(I(T'_{n,k}),A_{i-1}	) 
			\longrightarrow\ S/(I(T'_{n,k}),A_{i}	)  \longrightarrow\ 0,$$
			applying Lemma \ref{rle2} on this family of short exact sequences we get the following inequality			\begin{multline}\label{dn-ii}
				\depth (\mathbb{M}_{n,k})=\depth(S/(I(T'_{n,k}),A_{0})) \geq\\ \min\big\{\min\{\depth(S/((I(T'_{n,k}),A_{i-1}):x^{(1)}_{i})):i=1,2,\dots,n-1\},\depth(S/(I(T'_{n,k}),A_{n-1}))\big\}.
			\end{multline}
			We have the following isomorphism of $K$-algebras:
			$$	S/((I(T'_{n,k}),A_{i-1})	:x^{(1)}_{i})\cong \mathbb{M}_{n,k-2}^{(i-1)(n-1)}{\tensor_{K}}\mathbb{M}_{n,k-1}^{n-1-i}{\tensor_{K}} \mathbb{M}^{(n-1)^2}_{n,k-3}{\tensor_{K}}K[x^{(1)}_{i}],$$
			$$S/(I(T'_{n,k}),A_{n-1}) \cong \mathbb{M}_{n,k-2}^{(n-1)^2}\tensor_{K}K[x^{(0)}_{1}].$$ By applying Lemma \ref{LEMMA1.5}
			\begin{multline*}	\depth(S/((I(T'_{n,k}),A_{i-1}):x^{(1)}_{i}))=\\= \depth(\mathbb{M}_{n,k-2}^{(i-1)(n-1)})+\depth(\mathbb{M}_{n,k-1}^{n-1-i})+\depth(\mathbb{M}^{(n-1)^2}_{n,k-3})+\depth(K[x^{(1)}_{i}])
			\end{multline*} and 
			$$\depth(S/(I(T'_{n,k}),A_{n-1}))= \depth\big(\mathbb{M}^{(n-1)^2}_{n,k-2}\big)+\depth(K[x^{(0)}_{1}]).$$
			By Lemma \ref{nl} we have
			\begin{multline*}\label{eq3}
				\depth	(S/((I(T'_{n,k}),A_{i-1}):x^{(1)}_{i})) = (i-1)(n-1)\depth (\mathbb{M}_{n,k-2})\\+ (n-1-i)\depth(\mathbb{M}_{n,k-1})+(n-1)^2\depth(\mathbb{M}_{n,k-3})+1	\end{multline*}	
			and 
			\begin{equation*}
				\depth	(S/(I(T'_{n,k}),A_{n-1})) =(n-1)^{2}\depth(\mathbb{M}_{n,k-2})+1.
			\end{equation*}
			Now by Eq. \ref{dn-ii} we get
			\begin{multline}\label{eq2}
				\depth (\mathbb{M}_{n,k})\geq \min\Big\{\min_{i=1}^{n-1}\Big\{(i-1)(n-1)\depth (\mathbb{M}_{n,k-2})+ (n-1-i)\depth(\mathbb{M}_{n,k-1})\\+(n-1)^2\depth(\mathbb{M}_{n,k-3})+1\Big\},(n-1)^{2}\depth(\mathbb{M}_{n,k-2})+1\Big\}. 
			\end{multline}
			If $k=3$, then
			\begin{multline}\label{eqs}
				\depth (\mathbb{M}_{n,3})\geq\\ \min\Big\{\min_{i=1}^{n-1}\big\{(i-1)(n-1)\depth (\mathbb{M}_{n,1})+ (n-1-i)\depth(\mathbb{M}_{n,2})+(n-1)^2\depth(\mathbb{M}_{n,0})+1\big\},\\(n-1)^{2}\depth(\mathbb{M}_{n,1})+1\Big\}=\min\Big\{\min_{i=1}^{n-1}\big\{(i-1)(n-1)+ (n-1-i)(n-1)+(n-1)^2+1\big\},(n-1)^{2}+1\Big\}. 
			\end{multline}
			Since $$(i-1)(n-1)+ (n-1-i)(n-1)+(n-1)^2+1=(n-1)(2n-3)+1=(n-1)^2+1+\big((n-1)^2-(n-1)\big),$$ for all $i\in\{1,2,\dots,n-1\}$. Thus by Eq. \ref{eqs} $\depth(\mathbb{M}_{n,k})\geq (n-1)^2+1$. By Lemma \ref{nl1} we have $\depth(\mathbb{M}_{n,k})\leq (n-1)^2+1$, we get $\depth(\mathbb{M}_{n,k})= (n-1)^2+1$, as required.  Now let $k\geq 4$, we will prove the required result by induction on $k$. We consider three cases:
			\begin{description}
				
				\item[Case 1] Let $k\equiv 1(\mod 3)$. Recall Eq. \ref{eq1} \begin{equation*}
					\depth(\mathbb{M}_{n,k})\geq \min\{(n-1)^2\depth (\mathbb{M}_{n,k-2})+1,\,(n-1)\depth(\mathbb{M}_{n,k-1})\}.  
				\end{equation*}   
				Since $k\equiv 1(\mod 3)$ so $k-1\equiv 0\, (\mod \, 3)$ and $k-2\equiv 2\, (\mod \, 3)$ thus by induction on $k$, we get
				\begin{multline*}\label{coln22s2n}
					(n-1)^2\depth (\mathbb{M}_{n,k-2})+1= (n-1)^{2}\bigg( \frac{(n-1)^{(k-2)+2}-n+1}{(n-1)^{3}-1} \bigg)+1 = \frac{(n-1)^{k+2}-1}{(n-1)^{3}-1},
				\end{multline*}
				and 				
				\begin{multline*}
					(n-1)\depth(\mathbb{M}_{n,k-1})  = (n-1)\bigg( 	\frac{(n-1)^{2}((n-1)^{k-1}-1)}{(n-1)^{3}-1}+1\bigg)\\=\frac{(n-1)^{k+2}+n^{4}-5n^{3}+9n^{2}-8n+3}{(n-1)^{3}-1} = \frac{(n-1)^{k+2}-1}{(n-1)^{3}-1}+ \frac{n^{4}-5n^{3}+9n^{2}-8n+4}{(n-1)^{3}-1},
				\end{multline*}
				Note that $\frac{n^{4}-5n^{3}+9n^{2}-8n+4}{(n-1)^{3}-1}>0$ for all $n\geq3$ by using MATLAB\,\textsuperscript{\tiny\textregistered}. Thus $\depth(\mathbb{M}_{n,k})\geq \frac{(n-1)^{k+2}-1}{(n-1)^{3}-1}.$
				By Lemma \ref{nl1}, 
				$\depth(\mathbb{M}_{n,k})\leq(n-1)^{k-1}+\depth(\mathbb{M}_{n,k-3})$, since
				$k-3\equiv 1(\mod 3)$, thus by induction on $k$, we have 
				$\depth(\mathbb{M}_{n,k})\leq \frac{(n-1)^{k-1}-1}{(n-1)^{3}-1}+(n-1)^{k-1} = \frac{(n-1)^{k+2}-1}{(n-1)^{3}-1}.$
				Hence $ \depth (\mathbb{M}_{n,k})=\frac{(n-1)^{k+2}-1}{(n-1)^{3}-1},$ as required.		    
				
				\item[Case 2] Let $k\equiv 2(\mod 3)$. In this case $k-1\equiv 1\, (\mod \, 3)$ and $k-2\equiv 0\, (\mod \, 3)$ thus by induction on $k$, we have
				\begin{multline*}
					(n-1)^2\depth (\mathbb{M}_{n,k-2})+1= (n-1)^{2}\bigg( \frac{(n-1)^{2}((n-1)^{k-2}-1)}{(n-1)^{3}-1}+1\bigg)+1\\= \frac{(n-1)^{k+2}+n^{5}-6n^{4}+15n^{3}-20n^{2}+14n-5}{(n-1)^{3}-1}
					= \frac{(n-1)^{k+2}-n+1}{(n-1)^{3}-1}\\+ \frac{n^{5}-6n^{4}+15n^{3}-20n^{2}+15n-6}{(n-1)^{3}-1},
				\end{multline*}
				and 				
				\begin{equation*}
					(n-1)\depth(\mathbb{M}_{n,k-1})  = (n-1)\bigg(  \frac{(n-1)^{(k-1)+2}-1}{(n-1)^{3}-1}\bigg)=\frac{(n-1)^{k+2}-n+1}{(n-1)^{3}-1}.
				\end{equation*}
				Note that $\frac{n^{5}-6n^{4}+15n^{3}-20n^{2}+15n-6}{(n-1)^{3}-1}>0$ for all $n\geq 3$ by using MATLAB\,\textsuperscript{\tiny\textregistered}. Thus by using Eq. \ref{eq1}, we have
				$\depth (\mathbb{M}_{n,k})\geq \frac{(n-1)^{k+2}-n+1}{(n-1)^{3}-1}.$ For the other inequality we use again Lemma \ref{nl1}    $\depth(\mathbb{M}_{n,k}) \leq(n-1)^{k-1}+\depth(\mathbb{M}_{n,k-3})$. Since $k-3\equiv 2(\mod 3)$, so by induction on $k$,  $\depth(\mathbb{M}_{n,k})\leq (n-1)^{k-1}+ \frac{(n-1)^{(k-3)+2}-n+1}{(n-1)^{3}-1}= \frac{(n-1)^{k+2}-n+1}{(n-1)^{3}-1}.$  
				Hence $	\depth (\mathbb{M}_{n,k})=\frac{(n-1)^{k+2}-n+1}{(n-1)^{3}-1}.$
				\item [Case 3] Let $k\equiv 0(\mod 3)$. Recall Eq. \ref{eq2} 
				\begin{multline*}
					\depth (\mathbb{M}_{n,k})\geq \min\Big\{\min_{i=1}^{n-1}\Big\{(i-1)(n-1)\depth (\mathbb{M}_{n,k-2})+ (n-1-i)\depth(\mathbb{M}_{n,k-1})\\+(n-1)^2\depth(\mathbb{M}_{n,k-3})+1\Big\},(n-1)^{2}\depth(\mathbb{M}_{n,k-2})+1\Big\}. 
				\end{multline*}
				Since $ k-3\equiv 0\, (\mod \, 3)$, $ k-2\equiv 1\, (\mod \, 3)$ and  $ k-1\equiv 2\, (\mod \, 3)$ and $i\in\{1,2,\dots,n-1\}$. by induction on $k$, we have 
				\begin{multline*}
					(i-1)(n-1)\depth (\mathbb{M}_{n,k-2})+ (n-1-i)\depth(\mathbb{M}_{n,k-1})+(n-1)^2\depth(\mathbb{M}_{n,k-3})+1\\=(i-1)(n-1)\bigg( \frac{(n-1)^{(k-2)+2}-1}{(n-1)^{3}-1}\bigg) +(n-1-i)\bigg( \frac{(n-1)^{(k-1)+2}-n+1}{(n-1)^{3}-1}\bigg)\\+(n-1)^{2}\bigg( \frac{(n-1)^{2}((n-1)^{k-3}-1)}{(n-1)^{3}-1}+1\bigg)+1=\frac{(n-1)^{k+2}+n^{5}-6n^{4}+15n^{3}-21n^{2}+17n-7}{(n-1)^{3}-1} \\=\frac{(n-1)^{2}((n-1)^{k}-1)}{(n-1)^{3}-1}+1+ \frac{n^{5}-6n^{4}+14n^{3}-17n^{2}+12n-4}{(n-1)^{3}-1}
				\end{multline*}			
				and 
				\begin{multline*}
					(n-1)^{2}\depth(\mathbb{M}_{n,k-2})+1=
					(n-1)^2\bigg( \frac{(n-1)^{(k-2)+2}-1}{(n-1)^{3}-1}\bigg) +1=\frac{(n-1)^{2}((n-1)^{k}-1)}{(n-1)^{3}-1}+1.
				\end{multline*}
				Note that $\frac{n^{5}-6n^{4}+14n^{3}-17n^{2}+12n-4}{(n-1)^{3}-1}>0,$ for all $n\geq3$ by using MATLAB\,\textsuperscript{\tiny\textregistered}. Thus by Eq. \ref{eq2} we have $\depth(\mathbb{M}_{n,k})\geq \frac{(n-1)^{2}((n-1)^{k}-1)}{(n-1)^{3}-1}+1.$ Again by Lemma \ref{nl1} we have $\depth(\mathbb{M}_{n,k})\leq (n-1)^{k-1}+\mathbb{M}_{n,k-3}=(n-1)^{k-1}+\frac{(n-1)^{2}((n-1)^{k-3}-1)}{(n-1)^{3}-1}+1= \frac{(n-1)^{2}((n-1)^{k}-1)}{(n-1)^{3}-1}+1.$
			\end{description}
			This completes the proof.
			Proof for Stanley depth is similar.
		\end{proof}
		\begin{Corollary} \label{corsub}
			Let $n\geq 3$  and  $k\geq 1$. Then
			\begin{equation*} \label{pdimm}
				\pdim(\mathbb{M}_{n,k})=\left\{\begin{matrix}
					\frac{(n-1)^{k+1}-1}{n-2}-	\frac{(n-1)^{2}((n-1)^{k}-1)}{(n-1)^{3}-1}-1, & if \quad \quad k\equiv 0\, (\mod\, 3);\\\\
					\frac{(n-1)^{k+1}-1}{n-2}-\frac{(n-1)^{k+2}-1}{(n-1)^{3}-1}, & if \quad \quad k\equiv 1 \,(\mod\, 3);
					\\	\\\frac{(n-1)^{k+1}-1}{n-2}-\frac{(n-1)^{k+2}-n+1}{(n-1)^{3}-1}, &if \quad \quad k\equiv 2\, (\mod\, 3).
				\end{matrix}\right.
			\end{equation*}
		\end{Corollary}
		\begin{proof}
			The result follows by using Auslander–Buchsbaum formula {\cite[Theorems 1.3.3]{depth}} and Proposition \ref{leminitial}.
		\end{proof}
		\begin{Remark}
			{\em Let $n\geq 3$, we define $I(T_{n,0})=(0)$, thus $K[V(T_{n,0})]/I(T_{n,0})\cong K[x_1^{(0)}]$. We have  $\depth(K[V(T_{n,0})]/I(T_{n,0}))=  \sdepth(K[V(T_{n,0})]/I(T_{n,0}))=1$. }
		\end{Remark}
		\begin{Lemma}\label{nl1main}
			Let $n\geq3$, $k\geq 2$ and $R=K[V(T_{n,k})].$ If $k=2$, then  $$\depth(R/I(T_{n,2})),\sdepth(R/I(T_{n,2}))\leq n.$$ If $k\geq 3$, then $$\depth(R/I(T_{n,k}))\leq n(n-1)^{k-2}+\depth(K[V(T_{n,k-3})]/I(T_{n,k-3})), $$  and $$\sdepth(R/I(T_{n,k}))\leq n(n-1)^{k-2}+\sdepth(K[V(T_{n,k-3})]/I(T_{n,k-3})).$$  
		\end{Lemma}
		\begin{proof}
			We prove the result only for depth as the proof for Stanley depth is similar. Since  $u:=x_{1}^{(k-1)}x_{2}^{(k-1)}\cdots x^{(k-1)}_{n(n-1)^{k-2}}\notin I(T_{n,k})$, by Lemma  \ref{Cor7}, $\depth (R/I(T_{n,k}))\leq\depth (R/(I(T_{n,k}):u)).$ We have the following $K$-algebra isomorphisms, $R/(I(T_{n,2}):u)\cong K[L_1],$ and for $k\geq 3$,  $R/(I(T_{n,k}):u)\cong K[L_{k-1}]{\tensor_{K}} K[V(T_{n,k-3})]/I(T_{n,k-3}).$ Thus $\depth(R/(I(T_{n,2}):u))=n$ and for $k\geq 3$, by Lemma \ref{le3},   $\depth(R/(I(T_{n,k}):u))=n(n-1)^{k-2}+\depth(K[V(T_{n,k-3})]/I(T_{n,k-3})).$ This completes the proof. 
		\end{proof}

		\begin{Theorem}\label{thn}
			Let  $n\geq 3$ and  $k\geq 1.$ If $R=K[V({T_{n,k}})]$, then 
			
			\begin{equation*}
				\depth(R/I(T_{n,k}))=\sdepth(R/I(T_{n,k}))=\left\{\begin{matrix}
				\frac{n(n-1)^{k+1}+n^{3}-4n^{2}+4n-2}{(n-1)^{3}-1}, & if \quad \quad k\equiv 0\, (\mod\, 3);\\\\
					\frac{n(n-1)^{k+1}-n^{2}+2n-2}{(n-1)^{3}-1}, &if \quad \quad k\equiv 1 \,(\mod\, 3);
					\\	\\\frac{n(n-1)^{k+1}-n}{(n-1)^{3}-1}, &if \quad \quad k\equiv 2\, (\mod\, 3).
				\end{matrix}\right.
			\end{equation*}
			
		\end{Theorem}
		\begin{proof}
			First we prove the result for depth. If $k=1$, then $T_{n,1}$ is $(n+1)$-star  and by Lemma \ref{leAli} we have  $\depth(R/I(T_{n,1}))=1$ .
			Let $k\geq 2$. We have the following short exact sequence
			\begin{equation*}	0	\longrightarrow\ R/(I(T_{n,k}):x^{(0)}_{1}) \xrightarrow{\,\,.x^{(0)}_{1}\,}  R/I(T_{n,k}) 
				\longrightarrow\ R/(I(T_{n,k}),x^{(0)}_{1})  \longrightarrow\ 0.
			\end{equation*} Since
			$	 R/(I(T_{n,k}):x^{(0)}_{1}) \cong \mathbb{M}^{n(n-1)}_{n,k-2} {\tensor_{K}}K[x^{(0)}_{1}]$ and
			$ 	R/(I(T_{n,k}),x^{(0)}_{1}) \cong\mathbb{M}^{n}_{n,k-1}.$
			By Lemmas \ref{LEMMA1.5} and \ref{nl}, we have
			$$	\depth (R/(I(T_{n,k}):x^{(0)}_{1})) = n(n-1)\depth(\mathbb{M}_{n,k-2})+1 ,$$ and
			$$	\depth (R/(I(T_{n,k}),x^{(0)}_{1})) = n\depth (\mathbb{M}_{n,k-1}).$$	
			Thus by Lemma \ref{rle2}, 
			\begin{equation}\label{eq1main}
				\depth(R/I(T_{n,k}))\geq \min\{n(n-1)\depth(\mathbb{M}_{n,k-2})+1 ,n\depth (\mathbb{M}_{n,k-1})\}.  
			\end{equation}
			
			If $k=2$ then by using Eq. \ref{eq1main} and Proposition \ref{leminitial}		
			$$\depth(R/I(T_{n,2}))\geq \min\{n(n-1)\depth(\mathbb{M}_{n,0})+1 ,n\depth (\mathbb{M}_{n,1})\}=\min\{n(n-1)+1 ,n\}=n,$$ and by Lemma 	\ref{nl1main}, 
			$\depth(R/I(T_{n,2}))\leq n.$ Thus  $\depth(R/I(T_{n,2}))=n.$ This proves the result for $k=2.$ Let $k\geq 3.$ Consider the following short exact sequence
			\begin{equation*} 
				0	\longrightarrow\ R/(I(T_{n,k}):x^{(1)}_{1}	) \xrightarrow{\,\,.	x^{(1)}_{1}\,}  R/I(T_{n,k}) 
				\longrightarrow\ R/(I(T_{n,k}),x^{(1)}_{1}	)  \longrightarrow\ 0.
			\end{equation*}	
			We have		  $$   R/(I(T_{n,k}):x^{(1)}_{1}) \cong \mathbb{M}^{(n-1)^2}_{n,k-3}{\tensor_{K}}\mathbb{M}^{(n-1)}_{n,k-1} {\tensor_{K}}K[x^{(1)}_{1}],$$ 
			$$	R/(I(T_{n,k}),x^{(1)}_{1}) \cong \mathbb{M}_{n,k} {\tensor_{K}}\mathbb{M}^{(n-1)}_{n,k-2}.$$   
			By using Lemmas \ref{LEMMA1.5} and \ref{nl}, 			\begin{equation*}
				\depth (R/(I(T_{n,k}):x^{(1)}_{1})) = (n-1)^2	\depth (\mathbb{M}_{n,k-3})+(n-1)\depth (\mathbb{M}_{n,k-1})+1,
			\end{equation*}
			and
			$$	\depth (R/(I(T_{n,k}),x^{(1)}_{1})) =\depth (\mathbb{M}_{n,k})+ (n-1)\depth(\mathbb{M}_{n,k-2}).$$
			Thus by Lemma \ref{rle2}, 
			\begin{multline}\label{eq1main2}
				\depth(R/I(T_{n,k}))\geq \min\{(n-1)^2	\depth (\mathbb{M}_{n,k-3})+(n-1)\depth (\mathbb{M}_{n,k-1})+1 ,\\\depth (\mathbb{M}_{n,k})+ (n-1)\depth(\mathbb{M}_{n,k-2})\}.  \end{multline}
			If $k=3$ then by Eq. \ref{eq1main2}  and Proposition \ref{leminitial}		
			\begin{multline*}				\depth(R/I(T_{n,k}))\geq\\ \min\Big\{(n-1)^2	\depth (\mathbb{M}_{n,0})+(n-1)\depth (\mathbb{M}_{n,2})+1 ,\depth (\mathbb{M}_{n,3})+ (n-1)\depth(\mathbb{M}_{n,1})\Big\}\\=\min\{(n-1)^2+(n-1)^2+1 ,(n-1)^2+1+n-1\}\\=\min\{ (n-1)^{2}+n+(n-1)^{2}-n+1 ,(n-1)^{2}+n\}=(n-1)^{2}+n.  \end{multline*}	
			and by using Lemma \ref{nl1main} we have $\depth(R/I(T_{n,k}))\leq n(n-1)+1=(n-1)^2+n.$ Thus $\depth(R/I(T_{n,k}))= (n-1)^2+n,$ as required. Now let $k\geq 4.$
			We consider three cases.
			\begin{description}
				\item[Case 1] Let $ k\equiv 1\, (\mod \, 3)$.   Recall Eq. \ref{eq1main}
				\begin{equation*}
					\depth(R/I(T_{n,k}))\geq \min\{n(n-1)\depth(\mathbb{M}_{n,k-2})+1 ,n\depth (\mathbb{M}_{n,k-1})\}.  
				\end{equation*}
				
				Since $ k\equiv 1\, (\mod \, 3)$ so $ k-1\equiv 0\, (\mod \, 3)$ and $ k-2\equiv 2\, (\mod \, 3)$, thus by using Proposition \ref{leminitial}, we have
				\begin{equation*}
					n(n-1)\depth(\mathbb{M}_{n,k-2})+1
					= n(n-1)\bigg( \frac{(n-1)^{(k-2)+2}-n+1}{(n-1)^{3}-1} \bigg)+1=\frac{n(n-1)^{k+1}-n^{2}+2n-2}{(n-1)^{3}-1} ,
				\end{equation*} and 
				\begin{multline*}
					n\depth (\mathbb{M}_{n,k-1})= n\bigg( 	\frac{(n-1)^{2}((n-1)^{k-1}-1)}{(n-1)^{3}-1}+1\bigg)=\frac{n(n-1)^{k+1}+n^{4}-4n^{3}+5n^{2}-3n}{(n-1)^{3}-1}\\=\frac{n(n-1)^{k+1}-n^{2}+2n-2}{(n-1)^{3}-1}+\frac{n^{4}-4n^{3}+6n^{2}-5n+2}{(n-1)^{3}-1}.
				\end{multline*}
				Note that $\frac{n^{4}-4n^{3}+6n^{2}-5n+2}{(n-1)^{3}-1}>0$, for all $n\geq 3$ by using MATLAB\,\textsuperscript{\tiny\textregistered}. Thus we have $  
				\depth (R/I(T_{n,k}))\geq \frac{n(n-1)^{k+1}-n^{2}+2n-2}{(n-1)^{3}-1}.$
				By Lemma \ref{nl1main}, $\depth(R/I(T_{n,k}))\leq n(n-1)^{k-2}+\depth(K[V(T_{n,k-3})]/I(T_{n,k-3})),$ since $k-3\equiv 1\, (\mod \, 3),$ thus by induction on $k,$ we have $\depth(R/I(T_{n,k}))\leq n(n-1)^{k-2}+ \frac{n(n-1)^{(k-3)+1}-n^{2}+2n-2}{(n-1)^{3}-1}=\frac{n(n-1)^{k+1}-n^{2}+2n-2}{(n-1)^{3}-1}.$ 	Hence we get
				$\depth (R/I(T_{n,k}))=\frac{n(n-1)^{k+1}-n^{2}+2n-2}{(n-1)^{3}-1},$  as required. 
				\item[Case 2] 		
				Let $ k\equiv 2\, (\mod \,3)$.   
				In this case  $ k-2\equiv 0\, (\mod \, 3),$  $ k-1\equiv 1\, (\mod \, 3)$, by using Proposition \ref{leminitial}, we have 
				\begin{multline*}
					n(n-1)\depth(\mathbb{M}_{n,k-2})+1
					=   n(n-1)\bigg( \frac{(n-1)^{2}((n-1)^{k-2}-1)}{(n-1)^{3}-1}+1\bigg)+1\\= \frac{n(n-1)^{k+1}+n^{5}-5n^{4}+10n^{3}-11n^{2}+6n-2}{(n-1)^{3}-1}\\=\frac{n(n-1)^{k+1}-n}{(n-1)^{3}-1}+\frac{n^{5}-5n^{4}+10n^{3}-11n^{2}+7n-2}{(n-1)^{3}-1},
				\end{multline*}
				 	
				and	\begin{equation*}	n\depth (\mathbb{M}_{n,k-1})=n\bigg(  \frac{(n-1)^{(k-1)+2}-1}{(n-1)^{3}-1}\bigg)=\frac{n(n-1)^{k+1}-n}{(n-1)^{3}-1}.	\end{equation*} Note that $\frac{n^{5}-5n^{4}+10n^{3}-11n^{2}+7n-2}{(n-1)^{3}-1}>0,$ for all $n\geq 3$ by using MATLAB\,\textsuperscript{\tiny\textregistered}. By Eq. \ref{eq1main}, we get
				$\depth (R/I(T_{n,k}))\geq \frac{n(n-1)^{k+1}-n}{(n-1)^{3}-1}.$   
				For the other inequality, we again  use Lemma \ref{nl1main}, that is $\depth(R/I(T_{n,k}))\leq n(n-1)^{k-2}+\depth(K[V(T_{n,k-3})]/I(T_{n,k-3})),$ since $k-3\equiv 2\, (\mod \, 3),$ thus by induction on $k,$ we have
				$ 	\depth (R/I(T_{n,k}))\leq	 \frac{n(n-1)^{(k-3)+1}-n}{(n-1)^{3}-1}+n(n-1)^{k-2}= \frac{n(n-1)^{k+1}-n}{(n-1)^{3}-1}.$
				Hence	$	\depth (R/I(T_{n,k}))=\frac{n(n-1)^{k+1}-n}{(n-1)^{3}-1}.   $ 
				\item[Case 3] Let $ k\equiv 0\, (\mod \, 3)$. Recall Eq. \ref{eq1main2}
				\begin{multline*}
					\depth(R/I(T_{n,k}))\geq \min\{(n-1)^2	\depth (\mathbb{M}_{n,k-3})+(n-1)\depth (\mathbb{M}_{n,k-1})+1 ,\\\depth (\mathbb{M}_{n,k})+ (n-1)\depth(\mathbb{M}_{n,k-2})\}.  \end{multline*}
				Since $ k\equiv 0\, (\mod \, 3)$ so $ k-3\equiv 0\, (\mod \, 3)$, $ k-1\equiv 2\, (\mod \, 3)$ and  $k-2\equiv 1\, (\mod \, 3).$ By using Proposition \ref{leminitial}, we get 
				\begin{multline*}
					(n-1)^2	\depth (\mathbb{M}_{n,k-3})+(n-1)\depth (\mathbb{M}_{n,k-1})+1
					= (n-1)^{2}\bigg( \frac{(n-1)^{2}((n-1)^{k-3}-1)}{(n-1)^{3}-1}+1\bigg) +\\(n-1)\bigg( \frac{(n-1)^{(k-1)+2}-n+1}{(n-1)^{3}-1}\bigg)+1 =\frac{n(n-1)^{k+1}+n^{5}-6n^{4}+15n^{3}-21n^{2}+16n-6}{(n-1)^{3}-1}\\=\frac{n(n-1)^{k+1}+n^{3}-4n^{2}+4n-2}{(n-1)^{3}-1}+\frac{n^{5}-6n^{4}+14n^{3}-17n^{2}+12n-4}{(n-1)^{3}-1},
				\end{multline*}
				and

				\begin{multline*} \depth (\mathbb{M}_{n,k})+ (n-1)\depth(\mathbb{M}_{n,k-2})= \frac{(n-1)^{2}((n-1)^{k}-1)}{(n-1)^{3}-1}+1+(n-1)\bigg( \frac{(n-1)^{(k-2)+2}-1}{(n-1)^{3}-1}\bigg) \\=\frac{n(n-1)^{k+1}+n^{3}-4n^{2}+4n-2}{(n-1)^{3}-1}.	\end{multline*}
				Note that $\frac{n^{5}-6n^{4}+14n^{3}-17n^{2}+12n-4}{(n-1)^{3}-1}>0$, for all $n\geq3$ by using MATLAB\,\textsuperscript{\tiny\textregistered}. By Eq. \ref{eq1main2}, we get   $ 
				\depth(R/I(T_{n,k}))\geq \frac{n(n-1)^{k+1}+n^{3}-4n^{2}+4n-2}{(n-1)^{3}-1}.$
				Again by Lemma \ref{nl1main} and induction on $k$, we have  $\depth(R/I(T_{n,k}))\leq n(n-1)^{k-2}+\depth(K[V(T_{n,k-3})]/I(T_{n,k-3}))=\frac{n(n-1)^{(k-3)+1}+n^{3}-4n^{2}+4n-2}{(n-1)^{3}-1}+n(n-1)^{k-2}=\frac{n(n-1)^{k+1}+n^{3}-4n^{2}+4n-2}{(n-1)^{3}-1}.$ This completes the proof for depth. Proof for Stanley depth is similar.

			\end{description}
		\end{proof}

		\begin{Corollary} \label{corsub}
			Let  $n\geq 3$ and  $k\geq 1.$ If $R:=K[V({T_{n,k}})]$, then 
			\begin{equation*} \label{pdimm}
				\pdim(R/I(T_{n,k}))=\left\{\begin{matrix}
					\frac{n(n-1)^{k}-2}{n-2}-\frac{n(n-1)^{k+1}+n^{3}-4n^{2}+4n-2}{(n-1)^{3}-1}, & if \quad \quad k\equiv 0\, (\mod\, 3);\\\\	\frac{n(n-1)^{k}-2}{n-2}-\frac{n(n-1)^{k+1}-n^{2}+2n-2}{(n-1)^{3}-1}, & if \quad \quad k\equiv 1 \,(\mod\, 3)
					;\\	\\\frac{n(n-1)^{k}-2}{n-2}-\frac{n(n-1)^{k+1}-n}{(n-1)^{3}-1} ,&if \quad \quad k\equiv 2\, (\mod\, 3).
				\end{matrix}\right.
			\end{equation*}
		\end{Corollary}
		\begin{proof}
			The required result can be obtain by using Auslander–Buchsbaum formula {\cite[Theorems 1.3.3]{depth}} and Theorem \ref{thn}.
		\end{proof}

		\section{Regularity and Krull dimension}
		In this section, we  first compute regularity for cyclic module $\mathbb{M}_{n,k}$, after that we find out the regularity of $K[V(T_{n,k})]/I(T_{n,k})$. At the end, we compute  the Krull dimension for $K[V(T_{n,k})]/I(T_{n,k})$.

		\begin{Proposition} \label{lemaan}
			Let $n\geq 3$ and $k\geq 1$. Then
			
			\begin{equation*}
				\reg(\mathbb{M}_{n,k})=\left\{\begin{matrix}
					\frac{(n-1)^{k+2}-(n-1)^{2}}{(n-1)^{3}-1}, & if \quad \quad k\equiv 0\, (\mod\, 3);\\ \\ \frac{(n-1)^{k+2}-1}{(n-1)^{3}-1}, &if \quad \quad k\equiv 1 \,(\mod\, 3);
					\\	\\\frac{(n-1)^{k+2}-n+1}{(n-1)^{3}-1}, &if \quad \quad k\equiv 2\, (\mod\, 3).
				\end{matrix}\right.
			\end{equation*}
		\end{Proposition}
		
		\begin{proof}
			We will prove the result by induction on $k$. 	If $k=1$, then clearly $\indmat(T'_{n,1})=1,$ therefore  by  Lemma \ref{reg1}, we get  $\reg (\mathbb{M}_{n,1})=1.$ Let $k\geq 2$ and $S=K[V(T'_{n,k})].$ As we have noticed in Proposition \ref{leminitial},				$S/(I(T'_{n,k}):x^{(0)}_{1}) \cong \mathbb{M}^{(n-1)^2}_{n,k-2}\tensor_{K} K[x^{(0)}_{1}]$	and  $S/(I(T'_{n,k}),x^{(0)}_{1})\cong\mathbb{M}^{n-1}_{n,k-1}$. By 
			Lemmas \ref{le3} and \ref{nl} , we have
			\begin{equation}\label{Eq.11}		
				\reg(S/(I(T'_{n,k}):x^{(0)}_{1}))=\reg{(\mathbb{M}^{(n-1)^2}_{n,k-2}\tensor_{K} K[x^{(0)}_{1}])}=\reg{(\mathbb{M}^{(n-1)^2}_{n,k-2})} = (n-1)^2\reg (\mathbb{M}_{n,k-2}),	
			\end{equation}
			\begin{equation}\label{Eq.12}
				\reg (S/(I(T'_{n,k}),x^{(0)}_{1}))=\reg{(\mathbb{M}^{n-1}_{n,k-1})} = (n-1)\reg(\mathbb{M}_{n,k-1}).		\end{equation}	
		 If $k=2$, we have $\reg(S/(I(T'_{n,2}):x^{(0)}_{1}))=0$ 
		 and $\reg(S/(I(T'_{n,2}),x^{(0)}_{1}))=(n-1)\reg(\mathbb{M}_{n,1})=n-1.$ Hence by Lemma \ref{regul}(c), we have $\reg(\mathbb{M}_{n,2})=n-1.$ 
			For $k=3$, we have $\reg(S/(I(T'_{n,3}):x^{(0)}_{1}))=(n-1)^{2}\reg(\mathbb{M}_{n,1})=(n-1)^{2}$ and  $\reg(S/(I(T'_{n,3}),x^{(0)}_{1}))=(n-1)\reg(\mathbb{M}_{n,2})=(n-1)(n-1)=(n-1)^{2}.$ Therefore, by Lemma \ref{regul}(b), we have $\reg(\mathbb{M}_{n,3})\in \{(n-1)^{2}+1,(n-1)^{2}\}.$ Let us consider the following subsets of $E(T'_{n,3})$. 
			\begin{equation*}
				F_{1}= \underset{j=1}{\overset{n-1}{\bigcup}}\{\{x_{1}^{(2)},x_{j}^{(3)}\}\},
			\end{equation*} 
			\begin{equation*}
				F_{2}= \underset{j=(n-1)+1}{\overset{2(n-1)}{\bigcup}}\{\{x_{2}^{(2)},x_{j}^{(3)}\}\},
			\end{equation*} 
			
			\begin{center}	$\vdots$
			\end{center}	
			\begin{equation*}
				F_{(n-1)^{2}}= \underset{j=(n-2)(n-1)^{2}+(n-2)(n-1)+1}{\overset{(n-1)^{3}}{\bigcup}}\{\{x^{(2)}_{(n-1)^{2}},x_{j}^{(3)}\}\}.\end{equation*} 
			Clearly, $|F_i|=n-1$, for all $i$. Also each edge of $F_i$ for all $i$ has one vertex in $L'_{2}$ and other vertex in $L'_{3}$. Let $F':=\{e_{1},e_{2},\dots,e_{(n-1)^{2}}:e_i\in F_i\}$, it is easy to see that $F'$ is an induced matching and $|F'|=(n-1)^{2}.$
			If $F''$ is any induced matching such that it contains an edge between vertices of $L'_{1}$ and $L'_{2}$ or an edge between vertices $L'_0$ and $L'_{1}$, then it  leads to deletion of $(n-1)$ edges between $L_1'$ and $L_2'$ in $F''$. Since, $|L'_3|=(n-1)^3$, $|L'_2|=(n-1)^2$ and $|L'_1|=n-1$, thus we get $|F''|<|F'|.$  Hence $F'$ is a maximal induced matching and by Lemma \ref{reg1}, we have $\reg (\mathbb{M}_{n,3})=(n-1)^{2}$. Thus result follows for $k=3$. Let $k\geq 4$.
			We consider three cases:
			\begin{description}

				\item[Case 1] Let $ k\equiv 1(\, \mod \, 3).$  
			In this case $ k-2\equiv 2\, (\mod \, 3)$   and $ k-1\equiv 0\, (\mod \, 3)$,  by induction on $k$ and using Eqs. \ref{Eq.11} and \ref{Eq.12} we get
				\begin{equation*}
				    \reg(S/(I(T'_{n,k}):x^{(0)}_{1})) =(n-1)^{2}\bigg( \frac{(n-1)^{(k-2)+2}-n+1}{(n-1)^{3}-1}\bigg) =\frac{(n-1)^{k+2}-(n-1)^{3}}{(n-1)^{3}-1},
				\end{equation*}
				\begin{equation*}
					\reg(S/(I(T'_{n,k}),x^{(0)}_{1})) =(n-1)\bigg( \frac{(n-1)^{(k-1)+2}-(n-1)^{2}}{(n-1)^{3}-1}\bigg) =\frac{(n-1)^{k+2}-(n-1)^{3}}{(n-1)^{3}-1}.	\end{equation*}
				Since $\reg(S/(I(T'_{n,k}):x^{(0)}_{1}))=\reg(S/(I(T'_{n,k}),x^{(0)}_{1}))$, by	Lemma \ref{regul}(b),
				  $
				\reg(\mathbb{M}_{n,k}) \leq \reg(S/(I(T'_{n,k}),x^{(0)}_{1}))+1=\frac{(n-1)^{k+2}-1}{(n-1)^{3}-1}.$
				For the other inequality, let us define 
				\begin{equation}\label{inducematch}
					F_{(k-1,k)}= \underset{i=1}{\overset{(n-1)^{k-1}}{\bigcup}}\{\{x_{i}^{(k-1)},x_{(n-1)i}^{(k)} \}\},
				\end{equation} 
				and we have     $|F_{(k-1,k)}|=(n-1)^{k-1}.$
			Consider 
				$F=F_{(k-1,k)}\cup F_{(k-4,k-3)}\cup \dots \cup F_{(0,1)}.$ It is easy to see that $F$ is an induced matching.
			  Therefore,
				$	\indmat(T'_{n,k})\geq|F|,$
				where	  $				|F|=(n-1)^{k-1}+(n-1)^{k-4}+\dots+ (n-1)^{3}+(n-1)^{0}= \frac{(n-1)^{k+2}-1}{(n-1)^{3}-1}.     $
			By Lemma \ref{reg1}, we have
				$\reg (\mathbb{M}_{n,k}) \geq \frac{(n-1)^{k+2}-1}{(n-1)^{3}-1}.$
				Therefore, we get the required result.
				
				\item[Case 2] Let $ k\equiv 2\, (\mod \, 3).$ In this case $ k-2\equiv 0\, (\mod \, 3)$ and $ k-1\equiv 1\, (\mod \, 3)$. By Eqs. \ref{Eq.11} and \ref{Eq.12} and 	induction on $k$ we get
				\begin{equation*}
					\reg(S/(I(T'_{n,k}):x^{(0)}_{1}))= (n-1)^{2}\bigg( \frac{(n-1)^{(k-2)+2}-(n-1)^{2}}{(n-1)^{3}-1}\bigg)=\frac{(n-1)^{k+2}-(n-1)^{4}}{(n-1)^{3}-1},
				\end{equation*}
				and
				\begin{multline*}
					\reg(S/(I(T'_{n,k}),x^{(0)}_{1}))=(n-1)\bigg( \frac{(n-1)^{(k-1)+2}-1}{(n-1)^{3}-1}\bigg) =\frac{(n-1)^{k+2}-n+1}{(n-1)^{3}-1}\\=\frac{(n-1)^{k+2}-(n-1)^{4}}{(n-1)^{3}-1}+(n-1).\end{multline*} 
				Hence by Lemma \ref{regul}(c)  
				we have $\reg(\mathbb{M}_{n,k})=\frac{(n-1)^{k+2}-n+1}{(n-1)^{3}-1},$ as required.  
					\item[Case 3] Let $ k\equiv 0\, (\mod \, 3).$  
				If $T'_{n,k}=T'_{n,3}\cup 
	\underset{i=1}{\overset{(n-1)^{3}}{ \bigcup B_{i}}},$ where $B_{i}\cong T'_{n,k-3}$, $T'_{n,3}\cap B_{i}\neq \emptyset$ and  $B_{i}\cap B_{j}=\emptyset,$ for all $i\neq j.$ 
By Lemmas \ref{circulenttineq} and \ref{nl},  we have
			$$\reg(\mathbb{M}_{n,k})\leq \reg(\mathbb{M}_{n,3})+ \reg(\mathbb{M}^{(n-1)^{3}}_{n,k-3})= \reg(\mathbb{M}_{n,3})+ (n-1)^{3}\reg(\mathbb{M}_{n,k-3}).$$
		In this case $ k-3\equiv 0\, (\mod \, 3)$  	and by induction on $k$, we have
				\begin{equation*}
					\reg (\mathbb{M}_{n,k}) \leq (n-1)^{2}+ (n-1)^{3}\bigg( \frac{(n-1)^{(k-3)+2}-(n-1)^{2}}{(n-1)^{3}-1}\bigg)=\frac{(n-1)^{k+2}-(n-1)^{2}}{(n-1)^{3}-1}.\end{equation*}
			 	For the other inequality, we use Eq. \ref{inducematch} and define an induced matching     $F=F_{(k-1,k)}\cup F_{(k-4,k-3)}\cup \dots \cup F_{(2,3)}.$ 
			 As,
				$	\indmat(T'_{n,k})\geq|F|,$ 
				where	$   	
		|F|=(n-1)^{k-1}+(n-1)^{k-4}+\dots+ (n-1)^{5}+(n-1)^{2}= \frac{(n-1)^{k+2}-(n-1)^{2}}{(n-1)^{3}-1}.$	By Proposition \ref{reg1}(a), we have 
				$	\reg (\mathbb{M}_{n,k}) \geq \frac{(n-1)^{k+2}-(n-1)^{2}}{(n-1)^{3}-1}.$	Hence we get the required result. 
			\end{description}
		\end{proof}
		\begin{Theorem}
			Let  $n\geq 3$ and  $k\geq 1.$ If $R:=K[V({T_{n,k}})]$, then 
			\begin{equation*}
				\reg(R/I(T_{n,k}))=\left\{\begin{matrix}
					\frac{n(n-1)^{k+1}-n(n-1)}{(n-1)^{3}-1}, & if \quad \quad k\equiv 0\, (\mod\, 3);\\ \\ \frac{n(n-1)^{k+1}-(n-1)^{2}-1}{(n-1)^{3}-1}, &if \quad \quad k\equiv 1 \,(\mod\, 3);
					\\	\\\frac{n(n-1)^{k+1}-n}{(n-1)^{3}-1}, &if \quad \quad k\equiv 2\, (\mod\, 3).
				\end{matrix}\right.
			\end{equation*}
		\end{Theorem}
		
		\begin{proof}
			For $1\leq k\leq 3$, the result holds by a similar argument of  Proposition \ref{lemaan} and we have $\reg(R/(I(T_{n,1}))=1, \reg(R/(I(T_{n,2}))=n$ and  $ \reg(R/(I(T_{n,3}))=n(n-1).$
		 Let $k\geq 4.$ Since 			$R/(I(T_{n,k}):x^{(0)}_{1}) \cong \mathbb{M}^{n(n-1)}_{n,k-2}\tensor_{K} K[x^{(0)}_{1}]$	and  $R/(I(T_{n,k}),x^{(0)}_{1})\cong\mathbb{M}^{n}_{n,k-1}$. By 
		 Lemmas \ref{nl} and \ref{nl}, we have
			\begin{equation}\label{Eq.13}		
				\reg(R/(I(T_{n,k}):x^{(0)}_{1})) = n(n-1)\reg (\mathbb{M}_{n,k-2}),	
			\end{equation}
			\begin{equation}\label{Eq.14}
				\reg (R/(I(T_{n,k}),x^{(0)}_{1})) = n\reg(\mathbb{M}_{n,k-1}).		\end{equation}		 We consider three cases:
			\begin{description}

				\item[Case 1] Let $ k\equiv 1\, (\mod \, 3)$.  In this case  $ k-1\equiv 0\, (\mod \, 3)$, $ k-2\equiv 2\, (\mod \, 3),$ by Eqs. \ref{Eq.13} and \ref{Eq.14} and  using Proposition \ref{lemaan}, we get
				\begin{equation*}
					\reg (R/(I(T_{n,k}):x^{(0)}_{1})) = n(n-1)\bigg( \frac{(n-1)^{(k-2)+2}-n+1}{(n-1)^{3}-1}\bigg)=\frac{n(n-1)^{k+1}-n(n-1)^{2}}{(n-1)^{3}-1},	\end{equation*}
				\begin{equation*}
					\label{regmmm}
					\reg (R/(I(T_{n,k}),x^{(0)}_{1})) = n\bigg( \frac{(n-1)^{(k-1)+2}-(n-1)^{2}}{(n-1)^{3}-1}\bigg)=\frac{n(n-1)^{k+1}-n(n-1)^{2}}{(n-1)^{3}-1}.
				\end{equation*}
		Since $\reg(R/(I(T_{n,k}):x^{(0)}_{1}))=\reg(R/(I(T_{n,k}),x^{(0)}_{1}))$, by  Lemma \ref{regul}(b),	
				$\reg(R/I(T_{n,k})) \leq \reg(R/(I(T_{n,k}):x^{(0)}_{1}))+1=\frac{n(n-1)^{k+1}-n(n-1)^{2}+(n-1)^{3}-1}{(n-1)^{3}-1}=\frac{n(n-1)^{k+1}-(n-1)^{2}-1}{(n-1)^{3}-1}.$	For the other inequality, let us define
				\begin{equation}\label{maininduce}
					F_{(k-1,k)}= \underset{i=1}{\overset{ n(n-1)^{k-2}}{\bigcup}}\{\{x_{i}^{(k-1)},x_{(n-1)i}^{(k)}\}\},
				\end{equation} where
				$|F|=n(n-1)^{k-2} .$ 
	Consider 	$F=F_{(k-1,k)}\cup F_{(k-4,k-3)}\cup \dots \cup F_{(0,1)}$	be an induced matching.		 Therefore, $	\indmat(T_{n,k})\geq|F|, $	where   	$		|F|=n(n-1)^{k-2}+n(n-1)^{k-5}+\dots+ n(n-1)^{2}+1= \frac{n(n-1)^{k+1}-(n-1)^{2}-1}{(n-1)^{3}-1}. $  	By Lemma \ref{reg1}, 
				$	
				\reg(R/I(T_{n,k})) \geq \frac{n(n-1)^{k+1}-(n-1)^{2}-1}{(n-1)^{3}-1}.$
				Thus   $\reg(R/I(T_{n,k}))=\frac{n(n-1)^{k+1}-(n-1)^{2}-1}{(n-1)^{3}-1},$ as required.  
				\item[Case 2] Let $ k\equiv 2\, (\mod \, 3).$ 
			In this case $ k-1\equiv 1\, (\mod \, 3)$, $ k-2\equiv 0\, (\mod \, 3),$ by using Proposition \ref{lemaan} and Eqs. \ref{Eq.13}  and \ref{Eq.14} we get
				
				\begin{equation*}
					\reg (R/(I(T_{n,k}):x^{(0)}_{1})) = n(n-1)\bigg(\frac{(n-1)^{(k-2)+2}-(n-1)^{2}}{(n-1)^{3}-1}\bigg)=\frac{n(n-1)^{k+1}-n(n-1)^{3}}{(n-1)^{3}-1},
				\end{equation*} and	
				\begin{equation*}
					\reg (R/(I(T_{n,k}),x^{(0)}_{1})) = n\bigg( \frac{(n-1)^{(k-1)+2}-1}{(n-1)^{3}-1}\bigg)=\frac{n(n-1)^{k+1}-n(n-1)^{3}}{(n-1)^{3}-1}+n.\end{equation*} 
		
				Therefore by Lemma \ref{regul}(c), we have
				$\reg(R/I(T_{n,k}))=\frac{n(n-1)^{k+1}-n}{(n-1)^{3}-1},$   as required.
					\item[Case 3] Let $ k\equiv 0\, (\mod \, 3)$. Here $T_{n,k}=T_{n,3} \underset{i=1}{\overset{n(n-1)^{2}}{ \bigcup B_{i}}},$ where $B_{i}\cong T'_{n,k-3}$, $B_{i}\cap B_{j}=\emptyset$ and $T_{n,3}\cap B_{i}\neq \emptyset $ for all $i\neq j.$ 
By Lemmas \ref{circulenttineq} and \ref{nl},  we have
\begin{multline*}
     \reg(R/I(T_{n,k}))\leq\reg (K[V({T_{n,3}})]/I(T_{n,3}))+\reg (\mathbb{M}^{n(n-1)^{2}}_{n,k-3})\\ = \reg (K[V({T_{n,3}})]/I(T_{n,3}))+n(n-1)^{2}\reg (\mathbb{M}_{n,k-3}).
\end{multline*}
			In this case  $ k-3\equiv 0\, (\mod \, 3).$	As $\reg(R/I(T_{n,3}))=n(n-1)$, by  Proposition \ref{lemaan} we get
				\begin{equation*}
					\reg (R/I(T_{n,k})) \leq n(n-1)+ n(n-1)^{2}\bigg( \frac{(n-1)^{(k-3)+2}-(n-1)^{2}}{(n-1)^{3}-1}\bigg)=\frac{n(n-1)^{k+1}-n(n-1)}{(n-1)^{3}-1}.\end{equation*}
		For the other inequality,	we use  Eq. \ref{maininduce} and define    
				$F=F_{(k-1,k)}\cup F_{(k-4,k-3)}\cup \dots \cup F_{(2,3)}.$  
				It is easily seen that $F$ is an induced matching. Therefore,
				$	\indmat(T_{n,k})\geq|F|, $
				where    $		|F|=n(n-1)^{k-2}+n(n-1)^{k-5}+\dots +n(n-1)^{4}+n(n-1)= \frac{n(n-1)^{k+1}-n(n-1)}{(n-1)^{3}-1}.  $
			By Lemma \ref{reg1}, we have
				$		\reg (R/I(T_{n,k})) \geq \frac{n(n-1)^{k+1}-n(n-1)}{(n-1)^{3}-1}.$
				This completes the proof.
			\end{description}
		\end{proof}
		
		\begin{Lemma}\label{lem3k}
			Let $n\geq 3$ and $k\geq 1$. If $W\subset V(T_{n,k})$ be an independent set, then $W$ is maximal iff $ W\supseteq L_{q}$, for all $q$ with $k\equiv  q\,(\mod 2)$. \end{Lemma}
		\begin{proof} By definition of $T_{n,k}$, $E(T_{n,k})\cap\{\{u,v\}:u\in L_{q}, v\in L_{q-1}\}\neq \emptyset$ for all $1\leq q\leq k$. Also, $|L_0|=1$, $|L_1|=n$ and for $q\geq 2$, $|L_{q}|=(n-1)|L_{q-1}|.$ Thus $W\subset V(T_{n,k})$ is a maximal independent set, if and only if $W=L_{k}\cup L_{k-2}\cup\dots,\cup L_{k-2\lceil \frac{k-1}{2}\rceil}$, that is, if and only if $W\supseteq L_{q}$, for all $q$ with $k\equiv  q\,(\mod 2)$.
		\end{proof}
		
		\begin{Theorem}
			Let $n\geq 3$ and $k\geq 1$. If $R=K[V(T_{n,k})]$, then	$\dim(R/I(T_{n,k}))=\frac{(n-1)^{k+1}-1}{n-2}.$
		\end{Theorem}
		\begin{proof}
			By Lemma \ref{prok}, we know that  $\dim(R/I(T_{n,k}))=|W|$, where $W$ is a maximal independent set. Now by Lemma \ref{lem3k}, if $W$ is a maximal independent set then 
			\begin{equation*}
				W=\left\{\begin{matrix}
					L_k\cup L_{k-2} \cup \dots \cup L_0, & if\,\,\, \text{$k$ is even};\\ 
					\\L_k\cup L_{k-2} \cup\dots\cup L_1, & if \,\,\,\text{$k$ is odd}.
				\end{matrix}\right. 
			\end{equation*}
			Thus
			\begin{equation*}
				|W|=\left\{\begin{matrix}
					n(n-1)^{k-1}+n(n-1)^{k-3}+\dots+n(n-1)+1, & if \,\,\, \text{$k$ is even};\\ 
					\\n(n-1)^{k-1}+n(n-1)^{k-3}+\dots+n(n-1)^{2}+n, &if \,\,\, \text{$k$ is odd}.
				\end{matrix}\right. 
			\end{equation*}
			If $k$ is odd, then
			$$n+ n(n-1)^{2}+\dots +n(n-1)^{k-5} +n(n-1)^{k-3}+n(n-1)^{k-1}=\frac{(n-1)^{k+1}-1}{n-2},$$
			and if $k$ is even, then
			$$1+n(n-1)+\dots +n(n-1)^{k-5} +n(n-1)^{k-3}+	n(n-1)^{k-1}=\frac{(n-1)^{k+1}-1}{n-2}.$$

		\end{proof}

	\end{document}